\documentclass[reqno,a4paper]{amsart}

\usepackage{amsmath,amssymb,amsthm,amsfonts}
\usepackage{color}
\usepackage{mathrsfs,bbm}
\usepackage{hyperref}
\usepackage{enumerate} 

\newtheorem{theorem}{Theorem}[section]
\newtheorem{lemma}[theorem]{Lemma}
\newtheorem{prop}[theorem] {Proposition}
\newtheorem{cor}[theorem]  {Corollary}
\newtheorem{definition}[theorem] {Definition}

\theoremstyle{definition}

\theoremstyle{remark}
\newtheorem*{remark}{Remark}

\newcommand{\e}{\mathrm{e}} 
\newcommand{\N}{\mathbb{N}}
\newcommand{\R}{\mathbb{R}}

\newcommand{\C}{\mathbb{C}}

\newcommand{\E}{\mathbb{E}}
\newcommand{\dd}{\mathrm{d}} 
\newcommand{\eps}{\varepsilon}
\newcommand{\la}{\langle}
\newcommand{\ra}{\rangle}

\newcommand{\be}{\begin{equation}}
\newcommand{\ee}{\end{equation}}

\newcommand{\vect}[1]{\boldsymbol{#1}}

\newcommand{\set}[1]{{\left \{ #1 \right \}}}
\newcommand{\abs}[1]{{\left | #1 \right |}}
\newcommand{\bra}[1]{\left( #1 \right) }

\newcommand{\norm}[1]{\left\lVert #1 \right\rVert}

\newcommand{\one}{\mathbbm{1}}
\newcommand{\1}{\mathbbm{1}}

\newcommand{\dx}{\:\mathrm{d}}

\begin{document}

\title[Representations of the $su(1,1)$ current algebra] {Representations of the $su(1,1)$ current algebra and probabilistic perspectives}
\date{12 February 2024}

\author[S. Floreani]{Simone Floreani}
\address{Universit{\"a}t Bonn,
Institut für Angewandte Mathematik,
Abt. Stochastische Systeme mit Wechselwirkung,
Endenicher Allee 60, 53115 Bonn, Germany}
\email{sflorean@uni-bonn.de} 

\author[S. Jansen]{Sabine Jansen}
\address{Mathematisches Institut, Ludwig-Maximilians-Universit{\"a}t, 80333 M{\"u}nchen; Munich Center for Quantum Science and Technology (MCQST), Schellingstr. 4, 80799 M{\"u}nchen, Germany.}
\email{jansen@math.lmu.de}

\author[S. Wagner]{Stefan Wagner}
\address{Mathematisches Institut, Ludwig-Maximilians-Universit{\"a}t, 80333 M{\"u}nchen; Munich Center for Quantum Science and Technology (MCQST), Schellingstr. 4, 80799 M{\"u}nchen, Germany.}
\email{swagner@math.lmu.de}

\maketitle
	
\begin{abstract} 
	We construct three representations of the $su(1,1)$ current algebra: in extended Fock space, with Gamma random measures, and with  negative binomial (Pascal) point processes. For the second and third representations, the lowering and neutral operators are generators of  measure-valued branching processes (Dawson-Watanabe superprocesses) and spatial birth-death processes. The vacuum is the constant function $1$ and iterated application of raising operators yields Laguerre and Meixner polynomials.
In addition, we prove a Baker-Campbell-Hausdorff formula and give an explicit formula for the action of unitaries $\exp( k^+(\xi) - k^-(\xi))\exp(2 \mathrm i k^0(\theta))$ on exponential vectors. We explain how the representations fit in with a general scheme proposed by Araki and with representations of the $SL(2,\R)$ current group with Vershik, Gelfand and Graev's multiplicative measure.

	\medskip

    \noindent \emph{MSC2020:}  
81R10; 
33C45; 
60K35; 
60H40 

\medskip
    
    \noindent \emph{Keywords}: 
current algebras; infinite-dimensional Lie algebras; orthogonal polynomials; Gamma random measure; negative binomial (Pascal) point processes; Markov processes
\end{abstract}


\section{Introduction}

The generators for many interacting particle systems can be written with creation or annihilation operators, or raising and lowering operators. This holds true not only for bosons, fermions, or spin chains in quantum mechanics, but also for classical systems that evolve according to some continuous-time Markov process; this helps uncover useful hidden symmetries \cite{schutz1994non}. Sometimes the  generators of two different Markov semigroup ($P_t = \exp(tL)$, $t\geq0$) have the \emph{same} formal  expression with raising and lowering operators; the difference comes from different representations of the raising and lowering operators and their commutation relations. These observations lead to the algebraic approach to duality \cite{carinci-giardina-redig-book,giardina-kurchan-redig-vafayi2009,jansen-kurt2014,sturm2020algebraic}. 

A concrete example is the following model for energy transport, see Giardin{\`a}, Kurchan, Redig and Vafayi \cite{giardina-kurchan-redig-vafayi2009} and references therein. It is closely related to a model by Kipnis, Marchioro and Presutti \cite{kipnis-marchioro-presutti1982}. The simplest version of the \emph{Brownian momentum process} on a finite 1D lattice $\{1,2,\ldots, \ell \}$ is a continuous-time Markov process with state space $\R^\ell$ and formal generator 
\[
	L = \sum_{i,j} \Bigl(p_i \frac{\partial}{\partial p_j} - p_j \frac{\partial}{\partial p_i}\Bigr)^2.
\]
The sum is over nearest-neighbor pairs $\{i,j\}$, say with periodic boundary conditions. (More general versions allow for open systems interacting with reservoirs at the boundaries and for vector-valued momenta.) The interpretation is that at each site $i$, sits a particle with momentum $p_i$; two neighboring sites may exchange momentum but the total kinetic energy $\frac12 \sum_i p_i^2$ is conserved. The \emph{symmetric inclusion process} instead describes particles hopping on the 1D lattice. The state space is $\N_0^\ell$ and the formal generator is
\[
	\bigl(\tilde L f\bigr)(\vect n) = \sum_{i,j} n_i \bigl( n_j + \frac{1}{2}\bigr) \Bigl( f(n_1,\ldots, n_i - 1,\ldots, n_j +1,\ldots, n_\ell) - f(n_1,\ldots, n_\ell)\Bigr). 
\] 
Both $L$ and $\tilde L$ can be written as
\begin{equation} \label{eq:su11chain}
	\sum_{i,j} \Bigl( k_i^+ k_j^- + k_j^+ k_i^- - 2 k_i^0 k_j^0 + \frac 1 8\Bigr)
\end{equation}
with raising, lowering and neutral operators $k_i^\pm$, $k_j^0$ that satisfy the $su(1,1)$ commutation relations, 
\[
	[k_i^-,k_i^+] = 2 k_i^0,\quad [k_i^0, k_i^\pm] = \pm k_i^\pm.
\]
Operators associated with different lattice sites commute. 
For the Brownian energy process, 
\[
	k_i^+ = \frac12 p_i^2,\quad  k_i^- = \frac12  \frac{\partial^2}{\partial p_i^2},\quad  k_i^0 f = \frac14 \Bigl( \frac{\partial}{\partial p_i} p_i f+ p_i \frac{\partial}{\partial p_i} f\Bigr)
\] 
For the symmetric inclusion process,
\begin{align*}
	k_i^+ f(\vect n) & = n_i f(n_1, \ldots, n_{i-1}, n_i -1,n_{i+1}, \ldots, n_\ell),\\
	k_i^- f(\vect n) & = \Bigl( \frac 12 + n_i\Bigr) f(n_1, \ldots, n_{i-1}, n_i+1,n_{i+1}, \ldots, n_\ell),\\
	k_i^0 f(\vect n) & = \Bigl( n_i + \frac{1}{4} \Bigr) f(\vect n). 
\end{align*}
Notice that the lowering operator $k_i^-$ annihilates the indicator $\1_{\{\vect n =0\}}$ that there is no particle at all. 
The common algebraic formula~\eqref{eq:su11chain} helps explain dualities between the two processes \cite{giardina-kurchan-redig-vafayi2009}. In turn, duality helps investigate transport phenomena \cite{giardina-kurchan-redig2007,kipnis-marchioro-presutti1982}. 

How can we extend the algebraic formalism to systems on $\R^d$?  Infinite-di\-men\-sional Lie algebras and current algebras provide the proper algebraic framework; concretely, we look for representations of the $su(1,1)$ current algebra. Those keywords may suggest a daunting background from quantum field theory, algebra or geometry \cite{fuchs-book,kac-book}, but the definition we shall adopt (Definition~\ref{def:carep}), loosely following Araki \cite{araki1969} and Accardi and Boukas \cite{AccardiItoCalculus} is a simple variant of the familiar canonical commutation relations for bosons. In a previous article \cite{floreani-jansen-wagner2023algebraic}, we have already introduced a family of operators that form a representation of the $su(1,1)$ current algebra, and we have showed how to deduce intertwining relations for consistent Markov processes. Here we continue the analysis and place a stronger emphasis on the algebraic structure.

We give three representations of the $su(1,1)$ current algebra. The first representation is similar to the  representation in Fock space of the canonical commutation relations from quantum mechanics (Theorem~\ref{thm:fock-rep}). The other two representations are closer to Gaussian models for the Euclidean free fields, in which Hermite polynomials play a distinguished role \cite{glimm-jaffe}. The Hilbert spaces are $L^2$-spaces with respect to some probability measure, the vacuum is the constant function $1$, and iterated application of raising operators to the vacuum gives rise to orthogonal polynomials. Gaussian fields are replaced with Gamma random measures and negative binomial point processes (also called Pascal point processes). Hermite polynomials are replaced with Laguerre and Meixner polynomials (Theorem~\ref{thm:rep-laguerre} and~\ref{thm:rep-meixner}). 
%

In addition, we prove an infinite-dimensional version of the Baker-Campbell-Hausdorff formula (Theorem~\ref{thm:bch}) and we make explicit the action of a familiy of unitaries, analogous to the Weyl operators $\exp( \mathrm i (c(f) + c^\dagger (f))$ from quantum field theory, on exponential vectors (Theorem~\ref{thm:group}). The theorem should yield a representation of the current group for the universal cover  of $SU(1,1)$, but we refrain from making this explicit.

To the best of our knowledge, the three representations are new, however the ingredients that we use are not and there are closely related results in the literature. The first representation is similar to Boukas's representation of Feinsilver's finite-difference algebra \cite{boukas1988thesis} and {\'S}niady's representation of commutation relations for quadratic white noise \cite{sniady2000}. All these algebras are related:  Accardi and Skeide \cite{accardi-skeide2000} linked the representations of the finite-difference algebra and the renormalized square of white noise, which in turn is related to  to representations of the current algebra of $sl (2,\R)$ \cite{accardi2002renormalized}; finally, as is well known, the groups $SL(2,\R)$ and $SU(1,1)$ are isomorphic and so are their Lie algebras.

The extended Fock space used in our first representation also  appears in chaos decompositions in non-Gaussian white noise analysis, see Kondratiev, da Silva, Streit, and Us \cite{kondratiev-silva-streit-us1998}, Berezansky and Mierzejewski \cite{berezansky-mierzejewski2000}, Kondratiev and Lytvynov \cite{kondratiev-lytvynov2000gamma}. Gamma random measures and negative binomial (Pascal) point processes, together with raising, lowering, and neutral operators and infinite-dimensional Meixner and Laguerre polynomials have been intensely analyzed, see Lytvynov \cite{Lytvynov2003} and references therein. Lytvynov also made the connection with the algebra of the square of the white noise \cite{lytvynov2004square-of-white-noise}. For an algebraic stance on Pascal white noise, see Barhoumi, Ouerdiane and Riahi \cite{barhoumi-ouerdiane-riahi2008}. 

A slightly different series of articles goes back to Gelfand, Graev and Vershik \cite{gelfand-graev-vershik1971}. 
Araki \cite{araki1969} proposed a general scheme for finding factorizable representations of current groups in Fock spaces---but did not implement it for $SU(1,1)$ or $SL(2,\R)$. Vershik, Gelfand and Graev however did investigate those groups \cite{vershik-gelfand-graev1973}. For a full account and further references we refer the reader to Graev and Vershik \cite{vershik-graev2009}. We only mention this:  A key role is played by a kind of generalized infinite-dimensional Lebesgue measure, defined as a specific $\sigma$-finite measure that is absolutely continuous with respect to the law of a Gamma random measure. Yet another representation  involves instead ``quasi-Poisson spaces'' \cite{vershik-graev2011}. These representations are different from ours, but in Appendices~\ref{app:araki} and~\ref{app:lebesgue} we explain how our representations relate to Araki's scheme and to the infinite-dimensional Lebesgue measure.

All of the above-mentioned references have a strongly analytic flavor to them. This makes them hard to access for readers 
who may not wish to learn about distribution theory, nuclear spaces, rigged Hilbert spaces, or white noise analysis. Therefore we have strived to make the setup as simple as possible. We keep topological assumptions to a minimum and we allow for non-smooth test functions and reference measures with atoms. 

Another distinctive feature of our presentation is the emphasis on probabilistic interpretations of operators. It turns out that the lowering and neutral operators in our second and third representations generate continuous-time Markov processes that are of considerable intrinsic interest: measure-valued branching processes (Dawson-Watanabe superprocesses)  and spatial birth-death processes. This observation, to the best of our knowledge, is new, not only for the current algebra but also for the Lie algebra where the associated processes are diffusions (Laguerre and Bessel) and linear birth-death processes. 

The article is organized as follows. First we recall some basics about the group $SU(1,1)$, define the kind of representations that we shall study, and provide some context for our results (Section~\ref{sec:prelim}). Then we state our main results on representations of the $su(1,1)$ current algebra (Section~\ref{sec:results}). The simpler representations of the $su(1,1)$ Lie algebra are given in Section~\ref{sec:univariate}.  Finally we turn to the proofs of our main results (Section~\ref{sec:proofs}). In Appendices~\ref{app:araki} and~\ref{app:lebesgue} we explain how our approach fits into Araki's scheme \cite{araki1969} and to the representation of the $SL(2,\R)$ current group with generalized Lebesgue measure sketched in  Tsilevich, Vershik and Yor \cite{tsilevich-vershik-yor2001}.

\section{Preliminaries} \label{sec:prelim}

Here we gather some relevant background, most importantly,  we define the representations of $su(1,1)$ current algebras that we investigate in this article (Definition~\ref{def:carep}). In addition, we provide some background that aids the understanding of our main results: 
basics on the group  $SU(1,1)$; representations of the canonical commutation relations with Hermite and Charlier polynomials; how Gamma and and negative binomial (Pascal) laws naturally appear; and why we may expect relations with continuous-time Markov processes.

\subsection{The group \texorpdfstring{$SU(1,1)$}{SU(1,1)}}  \label{sec:22matrices}
The group $SU(1,1)$ consists of the complex $2\times 2$ matrices 
\begin{equation}\label{eq:gab}
	A= \begin{pmatrix}
			a & b \\
			\overline{b} & \overline{a}
	\end{pmatrix} 
\end{equation}
with $|a|^2 - |b|^2 = 1$. Every such matrix can be written as 
\begin{equation}\label{eq:arep}
   A(\xi,\theta) = \exp\bigl( \xi k^+ - \overline \xi k^-)\exp (2\mathrm i \theta k^0)
\end{equation} 
where $\xi \in \C$, $\theta \in \R$, and 
\[
	k^+ = \begin{pmatrix} 0 & \mathrm i \\ 0 & 0 \end{pmatrix},\quad 
	k^- = \begin{pmatrix} 0 & 0 \\ \mathrm i & 0 \end{pmatrix},\quad 
	k^0 = \begin{pmatrix} 1/2 & 0 \\ 0 & -1/2 \end{pmatrix}.
\] 
The matrices satisfy the commutation relations 
\begin{equation}\label{eq:su11algebra}
	[k^-,k^+] = 2 k^0,\quad [k^0,k^\pm] = \pm k^\pm.
\end{equation}
The $su(1,1)$ algebra consists of the complex linear combinations of the matrices $k^0,k^+,k^-$. The exponential in~\eqref{eq:arep} is easily computed, one has 
\begin{equation} \label{eq:axitheta}
	A(\xi,\theta) = \begin{pmatrix} 
			\cosh |\xi| & \mathrm i \frac{\xi}{|\xi|} \sinh |\xi| \\
			- \mathrm i  \frac{\overline{\xi}}{|\xi|} \sinh |\xi| & \cosh |\xi| 
	\end{pmatrix} 
	\begin{pmatrix}
		\e^{\mathrm i \theta} & 0 \\
		0 & \e^{-\mathrm i \theta}
	\end{pmatrix}.
\end{equation}
Bargmann \cite{bargmann1947} investigated representations of $SU(1,1)$ in the context of the Lorentz and Poincar{\'e} groups from special relativity. The $su(1,1)$ algebra also appears in quantum optics for pairs of photons: If $a$ and $a^\dagger$ are bosonic annihilation and creation operators, then $k^-= \frac12 a^2$, $k^+ =\frac12 (a^\dagger)^2$ and $k^0 = \frac12 a^\dagger a + \frac 14$ satisfy the $su(1,1)$ algebra \cite{brif-vourdas-mann1996,chiribella-dariano-perinotti2006}.  The operator $\exp(\frac12 (\xi a^2 - \overline{\xi} (a^\dagger)^2 ))$ is the \emph{squeeze operator} \cite[Chapter 2.7]{scully-zubairy-book}.

\subsection{Current algebra} As is customary with bosonic creation and annihilation operators, we work with operators indexed by functions. Let $(\mathbb X,\mathcal X)$ be a measurable space and $\alpha$ be a $\sigma$-finite reference measure on $\mathbb X$; for example, $\R^d$ with the Borel $\sigma$-algebra and the Lebesgue measure. 
We introduce an algebra of test functions
\begin{equation} \label{eq:c-algebra}
	\mathcal C = \bigl \{\varphi:\mathbb X\to \C:\, \text{ bounded, measurable, }\alpha(\{x:\varphi(x) \neq 0\})<\infty \bigr\}.
\end{equation}
The following definition is modeled after a similar definition in Accardi and Boukas \cite{accardi2009quantum} for the algebra of the renormalized square of white noise. 
Araki \cite{araki1969} defines the $su(1,1)$ current algebra over $\mathbb X$ as the set of $su(1,1)$-valued, bounded measurable functions, equipped with the pointwise Lie algebra operations.  Our function-indexed operators $k^\pm (\varphi), k^0(\theta)$ represent the matrix-valued maps $x\mapsto \varphi(x) k^+$, $x\mapsto \overline{\varphi(x)} k^-$, $x\mapsto \theta(x) k^0$ with $k^\pm, k^0$ the $2\times 2$-matrices from Section~\ref{sec:22matrices}. 

\begin{definition} \label{def:carep}
A family  $k^\pm(\varphi), k^0(\varphi)$, $\varphi \in \mathcal C$ of operators in a Hilbert space $\mathcal H$ is a \emph{Fock representation of the $su(1,1)$ current algebra} with \emph{vacuum} $\Psi$  if:
\begin{enumerate}[(i)]
	\item The operators are defined on a common domain $\mathcal D$ that is dense in $\mathcal H$ and they map $\mathcal D$ to itself, i.e., $k^\#(\varphi) \mathcal D \subset \mathcal D$. 
	\item $k^+(\varphi)$ and $k^0(\varphi)$ are linear in $\varphi$ while $k^-(\varphi)$ is antilinear in $\varphi$.
	\item For all $\varphi,\theta\in\mathcal C$,
	\begin{equation}
   		\begin{gathered}   \label{eq:commutation-relations-k}
    [k^-(\varphi), k^+(\theta)] = 2 k^0(\overline{\varphi} \theta)\\
     [k^0(\varphi), k^+(\theta)] =  k^+(\varphi \theta),
     \quad  [k^0(\varphi), k^-(\theta)] = - k^-(\overline{\varphi} \theta)
	    \end{gathered}
	  \end{equation}
		and $[k^\#(\varphi),k^\#(\theta)]=0$ for $\# \in \{0,+,-\}$.
	\item For all $f,g \in \mathcal D$ and $\varphi \in \mathcal{C}$, 
	    \begin{align} \label{eq:adjoint-relations-k}
		    \la f,  k^0(\varphi) g\ra = \la k^0(\overline{\varphi})f, g\ra, \quad \la f, k^+(\varphi)g\ra = \la k^-(\varphi) f,g \ra.
		 \end{align}
	\item $\Psi \in \mathcal{D}$, $||\Psi|| =1$ and $k^-(\varphi) \Psi =0$ for all $\varphi \in \mathcal C$.
	\item $\Psi$ is cyclic: the linear combinations of $\Psi$ and  $k^+(\varphi_1)\cdots k^+(\varphi_n)\Psi$ where $n\in \N$ and $\varphi_1,\ldots, \varphi_n \in \mathcal C$, are dense in $\mathfrak F$.  
\end{enumerate} 
\end{definition}

All our representations are \emph{factorizable}---when test functions have disjoint support, the associated operators commute. In addition, our three representations from Section~\ref{sec:results} satisfy
\[
	k^0(\varphi)\Psi =\Bigl( \frac1 2 \int \varphi \dd \alpha \Bigr) \Psi
\]
for all $\varphi \in \mathcal C$. Thus, the measure $\frac12 \alpha$ plays a role analogous to the Bargmann index for the discrete series, minimal weight representations of $SU(1,1)$ (see Section~\ref{sec:univariate}). 

We remark that the three representations are unitarily equivalent. More precisely, the isomorphism from \cite[Equation~(4.9)]{kondratiev-silva-streit-us1998} associated with the chaos decomposition of functionals of a Gamma random measure switches between the representation in extended Fock spaces and the representation outlined in Section \ref{sec:rep-gamma}. Similarly, the isomorphism from \cite[Corollary 5.3]{Lytvynov2003}, which relates to the chaos decomposition of functionals associated with a Pascal point process, switches between the representation in extended Fock spaces and the representation introduced in Section \ref{sec:rep-pascal}.

\subsection{Canonical commutation relations} 
The raising and lowering operators $k^+(f)$, $k^-(f)$ are similar to bosonic creation and annihilation operators $c^\dagger(f)$, $c(g)$ with the canonical commutation relations 
$[c(f), c^\dagger(g)] = \la f,g\ra$. The best known representation, in quantum many-body theory, is in the Fock space of sequences of symmetric functions $(f_n)_{n\in \N_0}$. For better comparison with our results, we recall two additional representations. To keep things concrete in this section $\mathbb X$ is $\R^d$ equipped with its Borel $\sigma$-algebra and with the Lebesgue measure $\lambda$. The scalar product in $L^2(\mathbb X, \lambda)$ is $\langle f, g\rangle = \int \overline f\, g \, \dd \lambda$. 

\subsubsection{Gaussian white noise, Hermite polynomials}
\label{sec:gauss-hermite}
The setting of the first additional representation is similar to that of the Euclidean quantum free field, we refer to Glimm and Jaffe \cite{glimm-jaffe} for background and proofs. By the Bochner-Minlos theorem, there is a unique probability measure $\mu$ on $\mathcal S'_\R$, the space of real-valued tempered distributions on $\R^d$, with 
\[
	\int \e^{\mathrm i \varphi(f)} \mu(\dd \varphi) = \exp\Bigl( -\frac12 \int f(x)^2 \dd x\Bigr).
\] 
for all real-valued Schwartz functions $f$. For all $n\in \N$ and real-valued Schwartz functions $f_1,\ldots, f_n$, the vector $(\varphi(f_1),\ldots,\varphi(f_n))$  is a Gaussian random vector with covariance matrix $C = (\int f_ig_j \dd \lambda)_{i,j=1,\ldots,n}$.
In particular 
\begin{itemize}
	\item If $f_1,\ldots,f_n$ have disjoint support then $\varphi(f_1)$,\ldots, $\varphi(f_n)$ are independent.
	\item Each variable $\varphi(f)$ has a normal distribution with mean zero and variance $\int f^2 \dd \lambda$. 
\end{itemize} 
The Hilbert space is $L^2(\mathcal S'_\R(\R^d),\mu)$, the space of complex-valued square-integrable functions. Let $\mathcal D$ be the set of functions $F:\mathcal S'_\R\to \C$ of the form 
\[
	F(\varphi) = P_n\bigl(\varphi(f_1),\ldots, \varphi(f_n)\bigr)
\] 
with $n\in \N$, $f_1,\ldots, f_n$ Schwartz functions, and $P_n:\C^n\to \C$ a multivariate polynomial. The annihilation operators are
\[
	c(f) F(\varphi) = \int \overline{f(x)} \frac{\delta}{\delta\varphi(x)} F(\varphi) \dd x
	=  \sum_{j=1}^n \la f, f_j\ra \frac{\partial P_n}{\partial x_j}\bigl(\varphi(f_1),\ldots,\varphi(f_n)\bigr)
\] 
and the creation operators are
\begin{align*}
	c^\dagger (g) F(\varphi) & = \varphi(g) F(\varphi) - \int g(x) \frac{\delta}{\delta \varphi(x)} F(\varphi) \dd x\\
	& = \varphi(g) F(\varphi) - \sum_{j=1}^n \la f_j, g\ra \frac{\partial P_n}{\partial x_j}\bigl(\varphi(f_1),\ldots,\varphi(f_n)\bigr).
\end{align*}
The vacuum is the constant function $\1$. Iterated application of creation operators to the vacuum gives rise to Hermite polynomials \cite{glimm-jaffe}. 

\subsubsection{Poisson point process. Charlier polynomials} \label{sec:poisson-charlier}
The second additional representation comes from Poisson white noise analysis (Kondratiev, da Silva, Streit, Us \cite{kondratiev-silva-streit-us1998} and references therein). It is similar to the previous representation except that the Gaussian measure $\mu$ is replaced by the law of a Poisson point process, supported on a simpler space of point configurations---there is no need for tempered distributions.

Let $\mathbf N$ be the set of finite or countable sums of Dirac measures, equipped with the $\sigma$-algebra generated by the counting variables $\eta\mapsto \eta(B)$, $B\in \mathcal X$. The law $\pi_\lambda$ of the \emph{Poisson point process with intensity measure $\lambda$} is the uniquely defined probability measure on $\mathbf N$ under which
\begin{itemize}
	\item Each counting variable $\eta\mapsto \eta(B)$, $B\in \mathcal X$, has a Poisson distribution with parameter $\lambda(B)$
	\[
		\pi_\lambda\bigl( \{\eta:\eta(B) = n\}\bigr) = \e^{-\lambda(B)} \frac{\lambda(B)^n}{n!}.
	\] 
	\item For all disjoint $B_1,\ldots, B_n\in \mathcal X$, the associated counting variables are independent. 
\end{itemize} 
Mathematical physicists should think of $\pi_\lambda$ as the grand-canonical Gibbs measure for a classical ideal gas in which the expected number of particles in a region $B$ is $\lambda(B)$. When $\lambda$ is Lebesgue measure, the ideal gas is homogeneous with density $1$. 

We work in the Hilbert space $L^2(\mathbf N,\pi_\lambda)$. 
Abbreviate $\eta(f) = \int f\dd \eta$ and let $\mathcal D$ be the space of polynomials in the variables $\eta(f)= \int f\dd \eta$. The annihilation and creation operators are \cite{kondratiev-silva-streit-us1998}
\begin{align*}
	c(f) F(\eta) & = \int \overline{f(x)} \bigl( F(\eta+ \delta_x) - F(\eta)\bigr) \lambda(\dd x),\\
	c^\dagger(f) F(\eta) & = \int f(x) F(\eta-\delta_x) \eta(\dd x) -  F(\eta) \int f\dd \lambda.  
\end{align*}
The vacuum $\Psi$ is again the constant function $1$, and iterated application of creation operators for indicators $f$ gives rise to \emph{Charlier polynomials} in the occupation numbers $\eta(B)$, $B\subset \mathbb X$. For example, for a finite-volume set  $B\subset \R^d$, 
\[
	\bigl(c^\dagger (\1_B) \bigr)^n \mathbf 1 (\eta) =  \mathscr C_n\bigl( \eta(B);\lambda(B) \bigr)
\] 
where $\mathscr C_n(x;\beta) = x^n+\text{lower order terms}$ is a univariate Charlier polynomial, orthogonal with respect to the measure $\nu(\{x\}) = \e^{-\beta} \beta^x/x!$ on $\N_0$ \cite{HypergeometricOrthogonalPolynomials}. 

\subsection{From Gauss and Poisson to Gamma and Pascal laws} \label{sec:quantum-proba}
Our second and third representations of the $su(1,1)$ current algebra rely on Gamma and Pascal laws. Here we briefly motivate how these distributions arise, building on results from quantum stochastic calculus.
Let $a^\dagger$ and $a$ be the creation and annihilation operators for the harmonic oscillator in some Hilbert space $\mathcal H$ with vacuum $\psi_0$, for example, $a = \frac{\dd}{\dd x}$ and $a^\dagger = x-\frac{\dd}{\dd x}$ in $L^2(\R,\exp(- x^2/2)\dd x)$, see the next subsection. 
Write $\la f(A) \rangle$ for the vacuum expectation values $\langle \psi_0,f(A)\psi_0\rangle$ of polynomial functions $f$ of an operator $A$.
The following is known: 
\begin{enumerate}[(i)]
	\item  The vacuum expectations of $a+a^\dagger$ are Gaussian: 
		\[
			\langle f(a+a^\dagger)\rangle = \frac{1}{\sqrt {2\pi}} \int_{-\infty}^\infty f(z) \exp( - z^2/2)\dd z
		\]
		for every polynomial $f:\R\to \R$ (just think of the harmonic oscillator).
	\item $a^\dagger + a + c a^\dagger a$, where $c\in \R\setminus \{0\}$, is associated with a Poisson law of  parameter $\lambda = 1/c^2$: 
	\[ 
		\langle  f(a+a^\dagger + c a^\dagger a) \rangle = \e^{-\lambda}\sum_{n=0}^\infty \frac{\lambda^n}{n!} f\bigl( c^{-1}(n-\lambda)\bigr). 
	\] 
	See Hudson and Parthasarathy \cite{hudson-parthasarathy1984} and Meyer \cite[Chapter IV.2(5)] {meyer2006quantum}. A related result is given by Ito and Kubo \cite[Proposition 5.2]{ito-kubo1988poisson-gauss}. 
 	\item $\frac12(a+a^\dagger)^2$ goes with a Gamma distribution: 
	\[
		\langle f\Bigl(\frac12 (a+a^\dagger)^2\Bigr) \rangle = \frac{1}{\sqrt \pi}\int_0^\infty f(x) x^{-1/2} \e^{-x} \dd x.
	\] 
	This follows from (i)---think $x = z^2/2$.
	\item $a^2 + (a^\dagger)^2+ \beta a^\dagger a$ with $\beta>2$ goes with negative binomial (Pascal) distribution with parameters $\alpha = 1/2$ and $p\in (0,1)$ the unique solution of $\sqrt p + {\sqrt p}^{-1} = \beta$: there exist non-zero reals $c_1,c_2$ such that
	\[
		\langle f\bigl(a^2 + (a^\dagger)^2+ \beta a^\dagger a \bigr) \rangle = (1-p)^{\alpha} \sum_{n=0}^\infty \frac{(\alpha)_n}{n!}\, p^n f\bigl( c_1 n + c_2 \bigr)
	\]
	where $(\alpha)_0 =1$, $(\alpha)_n= \alpha(\alpha+1)\cdots(\alpha +n-1)$. See Accardi, Franz and Skeide \cite{accardi2002renormalized} and the survey by Accardi and Boukas \cite[Section 17]{AccardiItoCalculus}. A similar result holds for $\beta < - 2$.
	\item $a^2+(a^\dagger)^2 + \beta a^\dagger a $ with $\beta \in (-2,2)$ is associated with a probability measure on $\R$, the orthogonality measure for the  Meixner-Pollaczek polynomials \cite{accardi2002renormalized}.
\end{enumerate} 
Remembering that $k^+ = \frac12 (a^\dagger)^2$, $k^- = \frac12 a^2$, $k^0 = \frac12 a^\dagger a + \frac14$ satisfy the commutation relations of the $su(1,1)$ algebra, we see that the operators 
\begin{equation} \label{eq:field-ops}
	X=k^+ + k^- + 2 k^0,\qquad Y=k^+ + k^- + ({\sqrt p}^{-1}+ \sqrt p) k^0
\end{equation}
with $p\in (0,1)$, 
are naturally associated with Gamma and negative binomial (Pascal) laws. 
Our three representations give priority to different operators: roughly, we diagonalize $k^0$, $X$ or $Y$. 

Combinations of the type~\eqref{eq:field-ops} are also motivated by the theory of orthogonal polynomials and Jacobi matrices, see Berezansky \cite{berezansky1998jacobi-fields}, Lytvynov \cite{lytvynov2003JFA} and references therein.
Consider for example the Laguerre polynomials $p_n(x) = \mathscr L_n^{(\alpha-1)}(x)$, orthogonal with respect to the measure $x^{\alpha-1} \exp(-x)\dd x$ on $\R_+$ and with leading coefficient $1$. They satisfy the three-term recurrence relation \cite[Eq.~(9.12.4)]{HypergeometricOrthogonalPolynomials}
\[
	x p_{n}(x) = p_{n+1}(x) + (2n+\alpha)p_n(x) + n(n+\alpha-1) p_{n-1}(x).
\]
Every polynomial $f(x)$ is a linear combination $f(x) = \sum_{n=0}^\infty \frac{1}{n!} g(n) p_n(x)$,  with at most finitely many non-zero  $g(n)$'s.  When multiplied with $x$, it is
\[
	x f(x) = \sum_{n=0}^\infty \frac1{n!} p_n(x)  \Bigl( n g(n-1) + (2n+\alpha) g(n) + (\alpha+ n) g(n+1)\Bigr).
\]
Thus multiplication by $x$ corresponds to a linear transformation of the sequence of coefficients, given by an infinite band-diagonal matrix, the Jacobi matrix. Let us define 
$J^+ g (n) = n g(n-1)$, $J^0 g(n) = (n+\alpha/2) g(n)$, $J^- g(n) = (\alpha+ n) g(n+1)$, then multiplication by $x$ corresponds to the action of $J^+ + 2 J^0 + J^-$ in the space of sequences. 

For infinite-dimensional orthogonal polynomials, the linear operators $J^\pm$ and $J^0$ are replaced by families of operators $J^\pm(\varphi)$ and $J^0(\varphi)$, indexed by test functions $\varphi$. The focus in this context is on the commutative family of operators $J^+(\varphi) + \lambda J^0(\varphi) + J^-(\varphi)$ (\emph{Jacobi fields}) and not so much on the full set of commutation relations (see, however, \cite{lytvynov2004square-of-white-noise}).  

\subsection{Continuous-time Markov processes} 
Finally we sketch how continuous-time Markov processes naturally enter the scene,  using the Gaussian model for the harmonic oscillator \cite[Chapter III.2(2)]{meyer2006quantum}.  Let $\mu=\mathcal N(0,1)$ be the standard normal distribution on $\R$. The norm in $L^2(\R,\mu)$ is 
\[
	||f||^2 = \frac{1}{\sqrt{2\pi}}\int_{-\infty}^\infty |f(x)|^2 \e^{- x^2/2} \dd x. 
\]
The densely defined operators $a  = \frac{\dd}{\dd x}$, $a^\dagger = x - \frac{\dd }{\dd x}$ satisfy $[a,a^\dagger] = \mathrm{id}$. Now, the point we wish to highlight is this: Both the annihilation operator $a$ and $L = - \widehat n$, with $\widehat n$ the number operator 
\[ 
	\widehat n = a^\dagger a= - \Bigl( \frac{\dd ^2}{\dd x^2} - x \frac{\dd }{\dd x}\Bigr) 
\]	
are generators of continuous-time Markov processes. The operator $L = - \widehat n$ generates an Ornstein-Uhlenbeck process for which the Gaussian measure $\mu$ is reversible. 
The annihilation operator generates deterministic motion to the right, $X_t = X_0 +t$. 	The creation operator corresponds to time-evolution of densities. Indeed $\exp(t a^\dagger)$ maps $f_0$ to $f_t(x) = f_0(x-t) \exp (t x)$. Thus, if  $X_0$ has probability density function $\rho_0(x)  = f_0(x) \exp( - x^2/2)/\sqrt{2\pi}$, then $X_t = X_0 + t$ has probability density function $\rho_t(x) = \rho_0(x-t) = f_t(x) \exp( - x^2/2)/\sqrt{2\pi}$.

We will encounter a similar structure for the second and third representations of the current algebra (and of the Lie algebra, see Section~\ref{sec:univariate}): The lowering and the neutral operator are related to continuous-time Markov processes, concretely spatial birth-death processes and measure-valued branching processes.

\section{Main results} \label{sec:results}

\subsection{Representation in extended Fock space} 
\label{sec:res-fock}

Let $(\mathbb X,\mathcal X)$ be a Borel space, for example, a complete metric separable space with its Borel $\sigma$-algebra. Let $\alpha$ be a $\sigma$-finite measure on $\mathbb X$. 

\subsubsection{Extended Fock space}   Define measures $\lambda_n$ on $\mathbb X^n$, equipped with the product $\sigma$-algebra $\mathcal X^n$, as follows. Let $\mathfrak S_n$ be the set of permutations of $\{1,\ldots,n\}$. Every permutation is a product of cycles with disjoint supports; the number of cycles is denoted $|\sigma|$. Given $f_n:\mathbb X^n\to \C$ a new function $f_{n}^\sigma:\mathbb X^{|\sigma|}\to \C$ is obtained by identifying variables that belong to the same cycle. For example, for $n=3$ and $\sigma = (1\ 3) (2)$, we get $f_3^\sigma(x,y) = f_3(x,y,x)$. The measure $\lambda_n$ on $\mathbb X^n$ is the uniquely defined measure for which 
\[
	\int f_n \dd \lambda_n  = \sum_{\sigma \in \mathfrak S_n} \int f_n^\sigma \dd \alpha^{|\sigma|} 
\] 
for all measurable $f_n:\mathbb X^n\to \R_+$. For example, we have $\lambda_1= \alpha$ and 
\[
	\int f_2 \dd \lambda_2 = \int f_2(x,y)\alpha(\dd x) \alpha (\dd y) + \int f(x,x) \alpha(\dd x). 
\] 
Let $L^2(\lambda_n) = L^2(\mathbb X^n,\mathcal X^n,\lambda_n)$ be the space of complex-valued, square-integrable functions $f_n$ and $L_\mathrm s^2(\lambda_n)$ be the subspace of symmetric functions, i.e., for every permutation $\sigma\in \mathfrak S_n$, 
$
	f_n(x_{\sigma(1)},\ldots, x_{\sigma(n)}) = f_n(x_1,\ldots,x_n)
$
for $\lambda_n$-almost all $(x_1,\ldots,x_n)$. The \emph{interacting} or \emph{extended Fock space} $\mathfrak F$ \cite{kondratiev-silva-streit-us1998}   consists of the sequences $(f_n)_{n\in \N_0}$ where $f_0\in \C$, $f_n\in L_\mathrm s^2(\lambda_n)$, and 
\[
	|f_0|^2+\sum_{n=1}^\infty \frac{1}{n!}\int |f_n|^2\dd \lambda_n<\infty. 
\] 
The extended Fock space is equipped with the scalar product 
\[
	\la (f_n)_{n\in \N_0}, (g_n)_{n\in \N_0}\ra = \overline{f_0} g_0 + \sum_{n=1}^\infty \frac{1}{n!} \int \overline{f_n} g_n\dd \lambda_n. 
\] 
The \emph{vacuum} is the sequence $\Psi = (1,0,0,\ldots)$. The symmetrized tensor product of functions $f_1,\ldots,f_n:\mathbb X\to \C$ is 
\[
	\bigl( f_1\otimes_\mathrm s\cdots \otimes_\mathrm s f_n\bigr)(x_1,\ldots,x_n) = \frac{1}{n!}\sum_{\sigma \in \mathfrak S_n} f_1(x_{\sigma(1)})\cdots f_{n}(x_{\sigma(n)}). 
\]
We shall often identify a function $F_n:\mathbb X^n\to \C$ with the sequence $(0,\ldots,F_n,0,\ldots)$ in $\mathfrak F$. With this identification in mind, the following lemma is useful. Remember the algebra $\mathcal C$ of bounded functions supported on sets of finite measure.

\begin{lemma} \label{lem:d-dense}
	The space $\mathcal D$ of finite linear combinations of the vacuum $\Psi$ and symmetrized tensor products $f_1\otimes_\mathrm s\cdots \otimes_\mathrm s f_n$ of functions $f_i\in \mathcal C$ is dense in $\mathfrak F$. 
\end{lemma} 

The lemma is proven in Section~\ref{sec:proof-extended-fock}.

\subsubsection{Exponential vectors} For $u\in \mathcal C$  with $\sup |u|<1$, we define a vector $\mathcal E_u \in \mathfrak F$ by 
\[
	(\mathcal E_u)_0 = 1,\quad (\mathcal E_u)_n(x_1,\ldots,x_n) = \prod_{i=1}^n u(x_i). 
\]
In terms of the raising operators introduced below, $\mathcal E_u = \exp( k^+(u))\Psi$, whence the name \emph{exponential vector} \cite[Chapter IV.1]{meyer2006quantum}. 
(Mathematical physicists may be more familiar with the notion of \emph{coherent states}; for bosons, these are just the  normalized exponential states, but for $SU(1,1)$ the relation is more subtle, see the remark after Theorem~\ref{thm:bch} below.)
We note 
\begin{equation}\label{eq:exp-scalar}
	\la \mathcal E_u,\mathcal E_v\ra = \exp\Bigl( - \int \log\bigl(1-\overline{u(x)} v(x)\bigr) \alpha(\dd x)\Bigr)
\end{equation}
where $- \log(1-z) = \sum_{j=1}^\infty z^j/j$ for $|z|<1$. Eq.~\eqref{eq:exp-scalar} already appeared in Boukas~\cite{boukas1988thesis,boukas1991}, however Boukas's Hilbert space is more abstract than our interacting Fock space.

 The linear combinations of exponential vectors are dense in $\mathfrak F$, moreover if $u_1,\ldots,u_n$ are linearly independent in $L^2(\mathbb X, \alpha)$, then the associated exponential vectors are linearly independent in $\mathfrak F$.  These properties are proven in the same way as their counterparts in the usual bosonic Fock space, we refer to Meyer \cite[Chapter IV.1.3]{meyer2006quantum}. 

\subsubsection{Representation of the $su(1,1)$ current algebra}
Define  \emph{raising, lowering and neutral operators} $k^+(\varphi), k^-(\varphi), k^0(\varphi):\mathcal D\to \mathfrak F$, indexed by $\varphi \in \mathcal C$, as follows: for $n\in \N$,
\begin{align*}
	\bigl(k^+(\varphi) f\bigr)_n(x_1,\ldots,x_n) & = \sum_{i=1}^n \varphi(x_i) f_{n-1}(x_1,\ldots,x_{i-1},x_{i+1},\ldots,x_n),\\
	\bigl(k^-(\varphi) f\bigr)_n(x_1,\ldots,x_n) & = (n+1) \int \overline{\varphi(y)} f_{n+1}(y,x_1,\ldots,x_n) \alpha(\dd y) \\
		&\qquad \qquad + \sum_{i=1}^n \overline{\varphi(x_i)} f_{n+1}(x_1,\ldots,x_n,x_i),\\
	\bigl(k^0(\varphi) f\bigr)_n(x_1,\ldots,x_n) &= \Bigl( \sum_{i=1}^n \varphi(x_i)+ \frac 12\int \varphi \dd \alpha \Bigr) f_n(x_1,\ldots,x_n).
\end{align*} 
For $n=0$ we use the same expressions, with $\sum_{i=1}^0 (\cdots) =0$. In particular, $(k^+(\varphi) f)_0=0$, $(k^-(\varphi) f)_0 = \int \overline{\varphi} f_1 \dd \alpha$, and $(k^0(\varphi) f)_0 = (\frac12 \int \varphi \dd \alpha) f_0$. 

Similar operators already appear in {\'S}niady \cite{sniady2000} and Lytvynov \cite{lytvynov2003JFA}. {\'S}niady studies them in the context of the renormalized algebra of the square of white noise, Lytvynov investigates infinite-dimensional orthogonal polynomials and relates Jacobi fields to the square of white noise in \cite{lytvynov2004square-of-white-noise}. Neither of them explicitly comments on the relation with the $su(1,1)$ current algebra. 

\begin{theorem} \label{thm:fock-rep}
	The operators $k^\pm(\varphi),k^0(\varphi):\mathcal D\to \mathfrak F$ form a Fock representation of the $su(1,1)$ current algebra in $\mathfrak F$ with vacuum $\Psi=(1,0,0,\ldots)$. 
\end{theorem}

The operators $k^\#(\varphi)$ may be written as
\[
	k^+(\varphi)  = a^\dagger(\varphi),\quad k^-(\varphi) = a(\varphi)+ b(\varphi),\quad k^0(\varphi) = \widehat n(\varphi) + \frac12 \int \varphi \dd \alpha
\]
where 
\begin{align*}
	a^\dagger(\varphi) f_1\otimes_\mathrm s \cdots \otimes_\mathrm s f_n & = (n+1) \varphi \otimes_\mathrm s f_1\otimes_\mathrm s \cdots \otimes_\mathrm s f_n,\\
	a(\varphi) f_1\otimes_\mathrm s \cdots \otimes_\mathrm s f_n & = \frac1 n \sum_{i=1}^n \la \varphi,f_i\ra f_1\otimes_\mathrm s \cdots \widehat{f_i} \cdots \otimes_\mathrm s f_n,\\
	b(\varphi) f_1\otimes_\mathrm s \cdots \otimes_\mathrm s f_n & = \frac 2 n \sum_{1 \leq i <j \leq n} \bigl(\overline{\varphi} f_i f_j\bigr)\otimes_\mathrm s f_1\otimes_\mathrm s \cdots \widehat f_i \cdots \widehat f_j\cdots \otimes_\mathrm s f_n,\\
	\widehat n(\varphi) f_1\otimes_\mathrm s \cdots \otimes_\mathrm s f_n & = \sum_{i=1}^n  f_1\otimes_\mathrm s \cdots  \otimes_\mathrm s (\varphi f_i) \otimes_\mathrm s\cdots  \otimes_\mathrm s f_n.
\end{align*}
Here, as usual, $\widehat f_i$ means omission of $f_i$. The scalar product is $\la \varphi, f\ra = \int\overline{\varphi} f \dd \alpha$. 
In $a^\dagger(\varphi)$, $a(\varphi)$ we recognize the bosonic creation and annihilation operators; $\widehat n(\varphi)$ is a smeared bosonic number operator (think $\int \varphi(x) a_x^\dagger a_x \, \dd x$).
(Factors of $\sqrt{n+1}$ or $1/\sqrt{n}$ often included in explicit formulas \cite[Chapter X.7]{reedSimonII} are missing because the scalar product in our (extended) Fock space has the factor $1/n!$.) The operators satisfy 
\begin{equation}\label{eq:ccr}
	[a(f),a^\dagger(g)] = \la f, g\ra,\quad [\widehat n(f), a^\dagger(g)] = a^\dagger(fg),\quad 
	[\widehat n(f), a(g)]= - a(f\overline g).
\end{equation}

\subsection{Unitary operators, Baker-Campbell-Hausdorff formula}

For $\xi,\theta \in \mathcal C$ with real-valued $\theta$, define unitary operators, analogous to~\eqref{eq:arep}, by
\begin{equation} \label{eq:udef}
	U(\xi,\theta) =\exp\bigl( k^+(\xi) - k^-(\xi) \bigr) \exp\bigl( 2\mathrm i k^0(\theta)\bigr).
\end{equation} 
The second exponential on the right is a multiplication operator, with multiplier $\exp( \mathrm i \int \theta \dd \alpha + 2\mathrm i \sum_{j=1}^n \theta(x_j))$. The first exponential is by definition $\exp(\mathrm i A)$ with $A$ the self-adjoint closure of $\frac 1 i (k^+(\xi) - k^-(\xi))$. The latter operator is essentially self-adjoint on $\mathcal D$; the proof with Nelson's analytic vector theorem is similar to \cite[Lemma 4.1]{floreani-jansen-wagner2023algebraic} and Reed and Simon \cite[Theorem X.41(a)]{reedSimonII} and therefore omitted. 

The unitaries $U(\xi,\theta)$ act on exponential vectors in a fairly explicit way. 
For $z:\mathbb X\to \{\zeta \in \mathbb C:\, |\zeta|<1\}$, let
\begin{equation} \label{eq:zxi}
	z_\xi(x) = \frac{z(x) + \frac{\xi(x)}{|\xi(x)|} \tanh|\xi(x)| }{1+  z(x) \frac{\overline{\xi(x)}}{|\xi(x)|} \tanh|\xi(x)|}
\end{equation}
and 
\begin{equation} \label{eq:cxi}
	C_\xi (z)= \exp\Biggl( - \int \log \Bigl(\cosh |\xi(x)| + z(x) \frac{\overline{\xi(x)}}{|\xi(x)|} \sinh |\xi(x)|\Bigr) \alpha(\dd x) \Biggr).
\end{equation} 
As explained by us in \cite[Section 5]{floreani-jansen-wagner2023algebraic}, the equations are related to the action of $SU(1,1)$ on the open unit disk by M{\"o}bius transforms, and to representations of $SU(1,1)$ given by Bargmann \cite[Section 9]{bargmann1947}. 

\begin{theorem} \label{thm:group}
	Let $\xi:\mathbb X\to \C$, $\theta:\mathbb X\to \R$, $z:\mathbb X\to \C$ be bounded functions in $\mathcal C$ with $\sup|z|<1$. Then 
	\begin{align*}
		\exp\bigl( k^+(\xi) - k^-(\xi)\bigr) \mathcal E_z & = C_\xi(z) \mathcal E_{z_\xi},\\
		\exp\bigl( 2\mathrm i k^0(\theta)\bigr) \mathcal E_z & = \exp\Bigl( \mathrm i \int \theta\, \dd \alpha\Bigr) \mathcal E_{\exp(2\mathrm i \theta)z}.
	\end{align*}
\end{theorem} 

The proof of the first formula is similar to Proposition~5.1 in \cite{floreani-jansen-wagner2023algebraic} and therefore omitted. The second formula is straightforward, we leave the details to the reader. 

The $2\times 2$ matrices in $SU(1,1)$ satisfy the Baker-Campbell-Hausdorff formula
\begin{equation} \label{eq:bch-univariate}
	\exp\Bigl( \xi k^+ - \overline{\xi} k^-\Bigr) = \exp\Bigl( \frac{\xi}{|\xi|} (\tanh|\xi|) k^+\Bigr) 
	\Bigl(\frac{1}{\cosh|\xi|}\Bigr)^{2k^0} \exp\Bigl( - \frac{\xi}{|\xi|} (\tanh|\xi|) k^-\Bigr).
\end{equation}
The unitaries $U(\xi,0)$ satisfy a similar equation. 

\begin{theorem}[Baker-Campbell-Hausdorff formula] \label{thm:bch}
	For all $\xi \in \mathcal C$ and all $f\in \mathcal D$, we have 
	\begin{equation} \label{equation: baker campbell hausdorff formula}
		\exp\Bigl( k^+(\xi) - k^-(\xi)\Bigr)  f
		= \exp\bigl(k^+(v)\bigr) \exp\bigl( k^0(w)\bigr) \exp\bigl( - k^-(v)\bigr) f,
	\end{equation} 
	where 
	$v := \1_{\{\xi\neq 0\}} \frac{\xi}{|\xi|} \tanh|\xi|$ and $w:=- 2\log \cosh|\xi|$.  
\end{theorem} 
On the right side $\exp(k^0(w))$ is a multiplication operator and $\exp(\pm k^\pm(v))$ are defined by exponential series that converge on the relevant domains (Lemma~\ref{lem:serconv}). The theorem is proven in Section~\ref{sec:bch-proof}. 

\begin{remark} 
	The analogue of the \emph{Perelomov coherent state}  \cite{brif-vourdas-mann1996,perelomov1972}  is 
	$\mathscr C_\xi = \exp\bigl( k^+(\xi) - k^-(\xi)\bigr) \Psi$.	Theorem~\ref{thm:bch} yields $\mathscr C_\xi = \lambda_\xi \exp(k^+(v))\Psi$ with $\lambda_\xi = \exp(- \int  \log (\cosh|\xi|)\, \dd \alpha)$ and $v$ as given in the theorem. Thus, the coherent state $\mathscr C_\xi$ differs from the exponential state $\mathcal E_v$ by a normalization and an additional reparametrization from $\xi$ to~$v$.
\end{remark} 

Araki \cite{araki1969} emphasizes the role of expectation functionals $g\mapsto E(g)$ for current groups. In light of this, we mention as a corollary an explicit expression for the vacuum expectations. 

\begin{cor}	[Vacuum expectations] 
	For all $\xi\in \mathcal C$ and all real-valued $\theta \in \mathcal C$, 
	\[
		\la \Psi, U(\xi,\theta) \Psi\ra = \exp\Bigl( \mathrm i \int \theta\, \dd \alpha - \int\bigl( \log \cosh|\xi|\bigr)\, \dd\alpha\Bigr). 
	\] 
\end{cor}

%
%
%
%
%

\subsection{Gamma random measure and Laguerre polynomials} \label{sec:rep-gamma} 

Next comes a representation in a probability space of random measures somewhat analogous to the Gaussian model for the canonical commutation relations from Section~\ref{sec:gauss-hermite} as the raising and lowering operators involve differential operators. 

Let $\mathbf M$ be the space of $s$-finite measures on $\mathbb X$ (a measure is $s$-finite if it is a countable sum of finite measures \cite{LastPenroseLecturesOnThePoissonProcess}). The space $\mathbf M$ is equipped with the $\sigma$-algebra generated by the variables $B\mapsto \mu(B)$ with $B\in \mathcal X$. Given the $\sigma$-finite reference measure $\alpha$, let $\rho_\alpha$ be the uniquely defined probability measure on $\mathbf M$ such that:
\begin{itemize}
	\item For all disjoint $B_1,\ldots, B_n\in \mathcal X$, the variables $\mu \mapsto \mu(B_i)$, $i=1,\ldots, n$, are independent. 
	\item For every $B\in \mathcal X$, the variable $\mu \mapsto \mu(B)$ has a Gamma distribution with shape parameter $\alpha(B)$ and scale parameter $1$: if $\alpha(B)<\infty$,
	\[
		\rho_\alpha \bigl( \{\mu:\mu(B) \leq y\}\bigr) =\frac{1}{\Gamma(\alpha(B))} \int_0^y t^{\alpha(B) - 1} \e^{-t} \dd t \quad (y\geq 0).
	\]
	When $\alpha(B) =0$ or $\infty$, we ask that $\mu(B) =0$ or $\infty$, $\rho_\alpha$-almost surely. 
\end{itemize}
The probability measure $\rho_\alpha$ is the law of a \emph{Gamma random measure}. It is the law of a compound Poisson process \cite[Chapter 15]{LastPenroseLecturesOnThePoissonProcess}: Let $\Pi = \sum_i \delta_{(X_i,Z_i)}$ be a Poisson point process on $\mathbb X \times (0,\infty)$ with intensity measure $\alpha(\dd x) z^{-1}\exp( -z) \dd z$, then 
$\xi = \sum_i Z_i \delta_{X_i}$ has law $\rho_\alpha$ \cite[Example 15.6] {LastPenroseLecturesOnThePoissonProcess}.

We work in the Hilbert space $L^2(\mathbf M, \rho_{\alpha})$. Let $\mathbb D$ be the space of polynomials in the variables $\mu(\varphi) = \int \varphi \dd \mu$, where $\varphi$ is in the algebra $\mathcal C$ of test functions from~\eqref{eq:c-algebra}. Thus $\mathbb D$ consists of the functions $F$ of the form 
\begin{equation} \label{eq:poly-gamma}
	F(\mu) = P_n\bigl(\mu(f_1),\ldots,\mu(f_n)\bigr)
\end{equation} 
with $n\in \N$, $f_1,\ldots, f_n\in \mathcal C$, and $P_n:\C^n \to \mathbb C$ a multivariate polynomial with complex coefficients. Notice that if $f\in \mathcal C$, then 
 $\mu(|f|)<\infty$ for $\rho_\alpha$-almost all $\mu \in \mathcal M$, thus $F(\mu)$ is well-defined except possibly on $\rho_\alpha$ null sets on which we may define it to be zero. 

\begin{lemma} \label{lem:polydense-gamma}
	The space $\mathbb D$ is dense in $L^2(\mathbf M,\rho_\alpha)$. 
\end{lemma}

The lemma is a classical infinite-dimensional variant of a result due to Riesz \cite{riesz1933}:  For probability measures on $\R$,  uniqueness in the moment problem implies density of polynomials. For the reader's convenience we give a self-contained proof in Section~\ref{sec:proofs-gamma}.

For a polynomial $F$ as in~\eqref{eq:poly-gamma}, $\mu \in \mathbf M$ and $x\in \mathbb X$, set
\begin{align*}
	\frac{\delta F}{\delta \mu(x)}(\mu)  &=  \frac{\dd}{\dd t} F(\mu+t\delta_x) \Big|_{t=0} = \sum_{j=1}^n f_j(x) \partial_j P_n\bigl( \mu(f_1),\ldots, \mu(f_n)\bigr),\\
	\frac{\delta^2 F}{\delta \mu(x)^2}(\mu)  &=  \frac{\dd^2}{\dd t^2} F(\mu+t\delta_x) \Big|_{t=0} = \sum_{i,j=1}^n f_i(x) f_j(x) \partial_i \partial_j P_n\bigl( \mu(f_1),\ldots, \mu(f_n)\bigr)
\end{align*}
where $\partial_j P_n$ is the partial derivative of $P_n$ with respect to the $j$-th variable. Finally, we define lowering, raising and neutral operators $K^\#(\varphi):\mathbb D\to \mathbb D$, where $\varphi \in \mathcal C$, by 
\begin{align*}
	K^-(\varphi) F(\mu) & =  \int \overline{\varphi(x)} \frac{\delta^2 F}{\delta \mu(x)^2} (\mu) \mu(\dd x) + \int \overline{\varphi(x)} \frac{\delta F}{\delta \mu(x)}(\mu) \alpha(\dd x) \\
	K^+(\varphi) F(\mu) &=  \int \varphi(x) \frac{\delta^2 F}{\delta \mu(x)^2} (\mu) \mu(\dd x) + \int \varphi(x) \frac{\delta F}{\delta \mu(x)}(\mu) (\alpha - 2\mu)(\dd x) \\
	& \qquad \qquad \qquad   + (\mu(\varphi) - \alpha(\varphi)) F(\mu) \\
	K^0(\varphi) F(\mu) & = - \int \varphi(x) \frac{\delta^2 F}{\delta \mu(x)^2} (\mu) \mu(\dd x) -  \int \varphi(x) \frac{\delta F}{\delta \mu(x)}(\mu) (\alpha- \mu)(\dd x) \\
	& \qquad \qquad \qquad  + \frac12 \alpha(\varphi) F(\mu).
\end{align*} 

\begin{remark} 
	For every real-valued test function $\varphi\in \mathcal C$, 
	\begin{equation}\label{eq:gamma-fieldop}
		K^+ (\varphi)+ K^- (\varphi)+ 2 K^0(\varphi) = \mu(\varphi).
	\end{equation}
	The right side is the multiplication operator with $\mu(\varphi)$. The above relation is in agreement with the univariate example (iii) from Section~\ref{sec:quantum-proba}. It should remind the reader of the definition of Segal\rq{}s field operator $\Phi_\mathrm S(f) $ as a sum of bosonic creation and annihilation operators \cite[Chapter X.7]{reedSimonII}.
\end{remark} 

For $\beta>0$, let $\mathscr L_n^{(\beta-1)}(x)$, $n\in \N_0$, be monic Laguerre polynomials, i.e., the polynomials of degree $n$ with leading coefficient $1$ that are orthogonal with respect to the measure $x^{\beta-1} \exp( - x)\dd x$ on $\R_+$. 

\begin{theorem} \label{thm:rep-laguerre}
	The operators $K^\pm(\varphi), K^0(\varphi):\mathbb D\to \mathbb D$, with $\varphi \in \mathcal C$, form a Fock representation in $L^2(\mathbf M,\rho_\alpha)$ of the $su(1,1)$ current algebra with vacuum $\Psi = \one$, the constant function $1$. Moreover, 
	\begin{equation} \label{eq:kplusit-laguerre}
		K^+(\1_{B_1})^{n_1}\cdots K^+(\1_{B_\ell})^{n_\ell} \1 (\mu)  =  \prod_{i=1}^{\ell} \mathscr L_{n_i}^{(\alpha(B_i)-1)} \bigl( \mu(B_i) \bigr)
	\end{equation}
	   for all $\ell\in \N$, disjoint $B_1,\ldots, B_\ell\in \mathcal{X}$ with $0 < \alpha(B_i) < \infty$, $n_1, \ldots, n_\ell \in \N$ and $\rho_{\alpha}$-almost all $\mu \in \mathbf{M}$.
\end{theorem} 

The theorem is proven in Section~\ref{sec:proofs-gamma}.

We remark that for non-negative $\varphi \in \mathcal{C}$, the operators have probabilistic interpretations: 
The lowering operator $K^-(\varphi)$ can be interpreted as the formal generator of a measure-valued Markov process that is a variant of superprocesses studied by Stannat \cite{stannat2003}. The process does not leave the distribution $\rho_{\alpha}$ invariant.
The raising operator $K^+(\varphi)$ describes the time-evolution of densities: if a process with generator $K^-(\varphi)$ has an initial law $\pi_0= F_0 \rho_{\alpha}$ that is absolutely continuous with respect to $\rho_{\alpha}$ with Radon-Nikodym derivative $F_0$, then the law at time $t$ has Radon-Nikodym derivative $\exp( tK^+(\varphi)) F_0$. 

Moreover, the process with formal generator
$- K^0(\varphi) + \frac{1}{2} \alpha(\varphi)$
is known as \emph{Dawson-Watanabe measure-valued continuous-state branching process}, see \cite[Equation~(1.1)]{transitionFunction} or \cite[Equation~(38)]{CanonicalCorrelations}.
If $\varphi$ is constant to one and $\alpha$ is a finite measure, the generator has spectrum $-\N_0$. 
Consequently, the corresponding Markov semigroup behaves analogously to the Ornstein-Uhlenbeck semigroup in the Gaussian case \cite[Equation~(1.67)]{MalliavinCalculus} or in the Poisson case  \cite[Proposition~4]{last2016}.

\subsection{Pascal point process and Meixner polynomials}
\label{sec:rep-pascal} 

We conclude with a representation in a probability space of point configurations analogous to the Poisson-Charlier representation for the canonical commutation relations from Section~\ref{sec:poisson-charlier}.

Let $\mathbf N = \mathbf N(\mathbb X)$ be the space of finite or countable sums of Dirac measures (including the zero measure) on the Borel space $(\mathbb X,\mathcal X)$. The space $\mathbf N$ is equipped with the $\sigma$-algebra $\mathcal N$ generated by the counting variables $B\mapsto \eta(B)$, $B\in \mathcal X$. For $a>0$ the rising factorial (Pochhammer symbol) is 
\[
	(a)_0 =1,\quad (a)_n = a (a+1)\cdots (a+n-1).
\] 
For $p\in (0,1)$ and $\alpha$ a $\sigma$-finite measure $\alpha$ on $\mathbb X$, let $\rho_{p,\alpha}$ be the distribution of a Pascal point process, i.e., $\rho_{p,\alpha}$ is the uniquely defined measure under which:
\begin{itemize}
	\item  The counting variables $\eta(B_1),\ldots, \eta(B_k)$ are independent, for all $k\geq 2$ and all pairwise disjoint $B_1,\ldots, B_k\in \mathcal E$. 
	\item For each $B\in \mathcal E$ with $0 < \alpha(B)< \infty$, 
	\[
		\rho_{p,\alpha}\bigl( \{\eta:\, \eta(B) =n\}\bigr) 
 = (1-p)^{\alpha(B)} \frac{p^n}{n!} \bigl( \alpha(B)\bigr)_n.
	\] 
 	If $\alpha(B) = 0$ or $\infty$ then $\eta(B) = 0$ or $\infty$ respectively, $\rho_{p,\alpha}$ almost surely. 
\end{itemize} 
We work in the Hilbert space $L^2(\mathbf N,\rho_{p,\alpha})$. Functions are complex-valued and the scalar product is $\la f,g\ra = \int \overline{f}g \, \dd \rho_{p,\alpha}$. Let $\mathbb D$ be the set of polynomials in the variables $\eta(\varphi) = \int \varphi \dd \eta$, i.e., $\mathbb D$ is the set of functions of the form 
\[
	F(\eta) = P_n\bigl( \eta(\varphi_1),\ldots, \eta(\varphi_n)\bigr)
\] 
with $n\in \N$, $\varphi_1,\ldots, \varphi_n\in \mathcal C$, and $P_n:\C^n\to \C$ a multivariate polynomial. 

\begin{lemma} \label{lem:polydense-pascal}
	The space $\mathbb D$ is dense in $L^2(\mathbf N,\rho_{p,\alpha})$. 
\end{lemma} 

The proof of the lemma is similar to the proof of Lemma~\ref{lem:polydense-gamma} and it is omitted. 
We introduce difference operators 
\[
	\mathrm D_x^+ F(\eta) = F(\eta+\delta_x) - F(\eta),\quad \mathrm D_x^- F(\eta) = F(\eta -\delta_x) - F(\eta)
\] 
and the constant $c = \sqrt{p}^{-1} - \sqrt p$. The backward difference $\mathrm D_x F(\eta)$ is only defined when $x$ belongs to the configuration $\eta$, i.e., $\eta(\{x\}) \geq 1$. For $F\in \mathbb D$, let
\begin{align}
	K^- (\varphi) F(\eta) & = \frac 1c \Bigl( \int \overline{\varphi(x)}\, \mathrm D_x^+ F(\eta)\, (\alpha + \eta)(\dd x)+ \int \overline{\varphi(x)}\, \mathrm D_x^- F(\eta)\, \eta(\dd x)\Bigr),\\
	K^+(\varphi) F(\eta) & = \frac 1c \Bigl( p  \int \varphi(x) F(\eta +\delta_x) (\eta+\alpha)(\dd x)
		+\frac 1 p \int \varphi(x) F(\eta- \delta_x) \eta(\dd x) \Bigr) \notag \\
			&\qquad \qquad \qquad - \frac 1c  \bigl( 2\eta(\varphi) + \alpha(\varphi)\bigr) F(\eta)  \\
	K^0(\varphi) F(\eta) &= - \frac1c \Bigl(\sqrt p \int \varphi(x)\, \mathrm D_x^+ F(\eta) \, (\alpha+\eta)(\dd x) +  \frac1{\sqrt p} \int \varphi(x)\, \mathrm D_x^- F(\eta)\, \eta(\dd x) \Bigr) \notag \\
	&\qquad \qquad \qquad + \frac 12\alpha(\varphi)  F(\eta). \label{eq:polyneutral}
\end{align} 

\begin{remark}	
	Let us denote the operator in the first line of~	\eqref{eq:polyneutral} by $\widehat N(\varphi)$ so that $K^0(\varphi) = \frac12 \alpha(\varphi) + \widehat N(\varphi)$. An explicit computation yields, for real-valued $\varphi\in \mathcal C$, 
	\[
		K^+(\varphi) + K^-(\varphi) + (\sqrt p +{\sqrt p}^{-1}) \widehat N(\varphi) 
		= c\, \eta(\varphi) - \sqrt{p}\, \alpha(\varphi).
	\] 
	 The equation is also similar to relations between quadratic white noise, Pascal laws and Meixner polynomials, see \cite[Section 17]{AccardiItoCalculus}, \cite[Section 4]{accardi2002renormalized} and case (iv) in Section~\ref{sec:quantum-proba}. 
\end{remark} 

Given $a>0$ and $p\in (0,1)$, let $\mathcal M_n(x;a,p)$ be the Meixner polynomials of degree $n$ with leading coefficients $1$, orthogonal with respect to the negative binomial law $\nu(\{x\}) = (1-p)^a (a)_x p^x/ x!$ on $\N_0$  \cite{HypergeometricOrthogonalPolynomials}.

\begin{theorem}  \label{thm:rep-meixner}
    The operators $K^+(\varphi)$, $K^0(\varphi)$, $K^-(\varphi)$, $\varphi \in \mathcal C$, are a Fock representation in $L^2(\mathbf N,\mathcal N,\rho_{p,\alpha})$ of the current algebra of $su(1,1)$ with vacuum $\one$, the constant function $1$. Moreover, with  $c = {\sqrt p}^{-1} - \sqrt p$,
    \begin{align}
        \label{eq:k+iterates}
        K^+(\one_{B_1})^{n_1}\cdots K^+(\one_{B_\ell})^{n_\ell} \one (\eta) = \prod_{j=1}^\ell c^{n_j} \mathcal M_{n_j} \bigl(\eta(B_j);\alpha(B_j),p\bigr)
    \end{align}
    for all $\ell\in \N$, disjoint $B_1,\ldots, B_\ell\in \mathcal{X}$ with $0 < \alpha(B_i) < \infty$, $n_1, \ldots, n_\ell \in \N$ and $\rho_{p, \alpha}$-almost all $\eta \in \mathbf{N}$.
\end{theorem} 

\noindent The theorem is proven in Section~\ref{sec:proofs-pascal}.

To conclude, we remark that for non-negative $\varphi$, the operators are amenable to a probabilistic interpretation. 
The processes that appear are variants of the linear birth-death processes on $\N_0$ called discrete Bessel and discrete Laguerre processes by Miclo and Patie \cite{miclo-patie2019}, see also  Section~\ref{sec:meixner-rep-uni} below. The role of Meixner polynomials and the Pascal law (among others) for linear birth-death processes was already studied by Karlin and McGregor \cite{karlin-mcgregor1958}, see also Chapter 3 in Schoutens \cite{schoutens}. 

The lowering operator $K^-(\varphi)$ is the generator of a spatial birth-and-death process.
The neutral operator $K^0(\varphi)$ is of the form $\frac12 \alpha(\varphi) - L(\varphi)$ and $L(\varphi)$ is  the generator of a spatial birth and death process with reversible measure  $\rho_{p,\alpha}$. Furthermore, when the measure $\alpha$ is finite so that the constant function $\one$ is in $\mathcal C$, one can show that $L(\one)$ has discrete spectrum $- \N_0$. 

\section{The univariate case: three representations of \texorpdfstring{$su(1,1)$}{su(1,1)}} \label{sec:univariate}

Our three representations of the $su(1,1)$ current algebra generalize much simpler univariate counterparts  that we recall here. 
The benefit is twofold.  First, we connect to standard representation theory for $su(1,1)$ and univariate orthogonal polynomials. Second, the probabilistic interpretation brings in well known processes: linear birth-death processes on $\N_0$ and diffusions (squared Bessel and Laguerre) on $\R_+$.

\subsection{Representation in weighted \texorpdfstring{$\ell^2$}{l2} space}  \label{sec:univar-weighted}
There are several classes of unitary irreducible representations of $SU(1,1)$. The representations studied here all belong to what is called the ``discrete class, minimal $m$, $D_k^+$'' \cite[\textsection 5.i]{bargmann1947} or ``principal discrete series $D_j^+$'' \cite[Section C]{lindblad-nagel1970}. Within that class, the representation is uniquely characterized by a half-integer $m_0\in \{\frac 1 2, 1, \frac 3 2,\ldots\}$, sometimes called \emph{Bargmann index}. The neutral operator $k^0$ has simple eigenvalues $m_0,m_0+1,\ldots$ and the Casimir operator $(k^0)^2 - \frac12 (k^+k^-+ k^-k^+)$, analogous to total spin, is $m_0(1-m_0)$ times the identity operator. 

The following is an example of such a representation. It is essentially taken from Giardin{\`a}, Kurchan, Redig and Vafayi \cite{giardina-kurchan-redig-vafayi2009}; it differs from the usual representations in $\ell^2(\N_0)$ by a similarity transformation. Remember the rising factorial $(\alpha)_0 =1$, $(\alpha)_n = \alpha(\alpha+1)\cdots (\alpha+n-1)$. 
Fix $\alpha>0$ and define weights $w_\alpha(n) = (\alpha)_n/n!$.  We work in the Hilbert space $\ell^2(\N_0,w_\alpha)$ and define operators  $k^\pm$, $k^0$ with common dense domain $d$, the space of sequences $(f(n))_{n\in \N_0}$ that have at most finitely many non-zero entries.
The operators are given by 
\[
	k^+ f(n) = n f(n-1),\quad k^-f(n) = (\alpha+n) f(n+1),\quad k^0 f(n) = \Bigl(n+\frac \alpha 2\Bigr) f(n). 
\] 
Equivalently, with $e_n(j) = \delta_{n,j}$, 
\begin{equation} \label{eq:ken}
	k^+ e_n = (n+1) e_{n+1},\quad k^- e_n = (\alpha+n-1) e_{n-1},\quad k^0 e_n = \Bigl(n+\frac \alpha 2\Bigr)e_n.
\end{equation} 
and $k^- e_0 =0$. One easily checks that $k^\pm$ and $k^0$ map $d$ into itself and on $d$ they satisfy the commutation relations~\eqref{eq:su11algebra}. Moreover, $k^0$ is symmetric on $d$ and $\la f, k^+ g\ra = \la k^- f, g\ra$ for all $f,g \in d$.

When $\alpha$ is integer, we get a representation of the group $SU(1,1)$ with Bargmann index $m_0 = \alpha/2$. 
The representation maps  the $2\times 2$ matrix $A(\xi,\theta) \in SU(1,1)$ from Eq.~\eqref{eq:axitheta}  to the unitary operator $U(\xi,\theta) = \exp(\xi k^+ - \overline \xi k^-) \exp( 2\mathrm i \theta k^0)$.  When $\alpha$ is not integer, we get instead a representation of the universal covering group of $SU(1,1)$ \cite{ruehl-yunn1976}. (Notice that $A(\xi,\theta+2\pi) = A(\xi,\theta)$ no matter what, but $U(\xi,\theta+ 2\pi) = U(\xi,\theta)$ only if $\alpha$ is integer.)

\subsection{Representation with Laguerre polynomials} \label{sec:laguerre-rep-uni}

Let $\beta>-1$. The \emph{Laguerre polynomials} in hypergeometric form are given by \cite[Chapter 9.12]{HypergeometricOrthogonalPolynomials}
\[
	L_n^{(\beta)}(x) = \frac{(\beta+1)_n}{n!} \sum_{j=0}^n \frac{(-n)_j}{(\beta+1)_j} \frac{ x^j}{j!} 
	= \frac{(-x)^n}{n!}+ O(x^{n-1})
\]
they satisfy the orthogonality relation,
\begin{equation} \label{eq:ortho-laguerre}
	\frac{1}{\Gamma(\beta+1)} \int_0^\infty L_n^{(\beta)}(x)L_m^{(\beta)}(x)  x^\beta \e^{-x}  \dd x
	= \delta_{m,n} \frac{(\beta+1)_n}{n!}. 
\end{equation} 
For $\alpha>0$, let $\ell^2(\N_0,w_\alpha)$ be the Hilbert space with weights $w_\alpha(n) = (\alpha)_n/n!$ from Section~\ref{sec:meixner-rep-uni}. Let $\mu_\alpha$ be the probability measure on $\R_+$ with density $\Gamma(\alpha)^{-1} x^{\alpha-1} \exp(-x)$ and $L^2(\R_+,\mu_\alpha)$ be the space of Borel-measurable square-integrable functions (complex-valued). By the orthogonality relation~\eqref{eq:ortho-laguerre}, the linear operator 
\[
	U:\ell^2(\N_0,w_\alpha)\to L^2(\R_+,\mu_\alpha),\quad f \mapsto Uf = \sum_{n=0}^\infty f(n) L_n^{(\alpha-1)}
\] 
is an isometry that maps $e_n$ to $L_n^{(\alpha-1)}$ and in particular $e_0$ to the constant function $\1$. The linear hull of $e_0,\ldots,e_n$ is mapped to the space of polynomials of degree at most $n$. 
Upon setting $K^\# = U^{-1} k^\# U$ we obtain a representation of the $su(1,1)$ algebra that satisfies 
$K^- \1 =0$ and
\begin{equation} \label{eq:kenlag}
	K^+ L_n = (n+1) L_{n+1},\ K^- L_n= (\alpha+n-1) L_{n-1},\quad K^0 L_n = \Bigl(n+\frac \alpha 2\Bigr)K^0 L_n,
\end{equation} 
compare~\eqref{eq:ken}. For better readability we have dropped the superscripts on the Laguerre polynomials. The following proposition identifies $K^\#$ as differential operators; it is similar to Proposition~4.13 in Groenevelt \cite{Groenevelt2019}.

\begin{prop} \label{prop:rep-uni-laguerre}
	We have
	\begin{align*}
		K^+& = - x \partial_x^2 + (2x-\alpha)\partial_x + (\alpha-x),\\
		K^- & = -  x \partial_x^2 - \alpha \partial_x,\\
		K^0 &= \frac{\alpha}{2} - \bigl( x\partial_x^2 + (\alpha-x)\partial_x\bigr). 
	\end{align*}
\end{prop} 

Notice that both $L = (\frac \alpha 2 - K^0)$ and $-K^-$ are infinitesimal generators of diffusions on $\R_+$. The operator $L$ is the generator of the \emph{Laguerre semigroup} \cite[Section~2.7.3]{BGL} which is the Markov semigroup of the \emph{Cox-Ingersoll-Ross process} (CIR) in mathematical finance to model interest rate movements \cite{CIR}. It has the Gamma distribution $\mu_\alpha$ as a reversible measure. The lowering operator corresponds to a squared Bessel process. 

\begin{proof}
	The generating function of our Laguerre polynomials is \cite[Eq.~(9.12.10)]{HypergeometricOrthogonalPolynomials}
	\[
		G_t(x) = \sum_{n=0}^\infty t^n L_n(x) = (1-t)^{-\alpha} \exp\Bigl( \frac{xt}{t-1}\Bigr). 
	\] 
	In view of Eq.~\eqref{eq:kenlag}, we have 
	\begin{align*}
		\bigl( K^0 - \frac{\alpha}{2}\bigr)G_t(x) & = \sum_{n=0}^\infty t^n n L_n(x) 
			= t\partial_t G_t(x) \\
		  & = -  \Bigl(\alpha \frac{t}{t-1} + \frac{xt }{(t-1)^2} \Bigr) G_t(x) \\
		  & = - \Bigl( x \frac{t^2}{(t-1)^2} + (\alpha - x) \frac{t}{t-1}\Bigr) G_t (x)\\
		  & = - (x\partial_x^2 + (\alpha-x)\partial_x) G_t(x). 
	\end{align*}
	This holds true for all sufficiently small $t$, the differential expression for $K^0$ follows. For the raising operators:
	\begin{align*}
		 K^+ G_t(x) & = \sum_{n=0}^\infty t^n (n+1)  L_{n+1}(x) = \sum_{n=1}^\infty n t^{n-1} L_n(x) = \partial_t G_t (x) \\ 
		 & = \Bigl( - \frac{\alpha}{t-1} - \frac{x}{(t-1)^2}\Bigr) G_t(x)  \\
		 & = \Bigl( - x\frac{t^2}{(t-1)^2} + (2x-\alpha)\frac{t}{t-1} +\alpha -x\Bigr) G_t(x) \\
		 & = \Bigl( - x \partial_x^2 + (2x-\alpha)\partial_x + (\alpha-x)\Bigr) G_t(x).
	\end{align*}
	Finally 
	\begin{align*}
		K^-G_t(x) &= \sum_{n=1}^\infty t^n (\alpha+n-1) L_{n-1}(x) \\
		& = \alpha t G_t(x) + t^2 \partial_t G_t(x) \\
		& = \Bigl( \alpha t - \alpha \frac{t^2}{t-1} - x\frac{t^2}{(t-1)^2}\Bigr) G_t(x)\\
		& =\Bigl( - \alpha \frac{t}{t-1} - x \frac{t^2}{(t-1)^2}\Bigr) G_t(x)\\
		& = - \bigl( x \partial_x^2 + \alpha \partial_x\bigr) G_t(x). \qedhere
	\end{align*}	
\end{proof} 

\subsection{Representation with Meixner polynomials} \label{sec:meixner-rep-uni}
For $p\in (0,1)$ and $\alpha>0$, the Meixner polynomials are given by \cite[Chapter 9.10]{HypergeometricOrthogonalPolynomials}
\[
	M_n(x;\alpha,p) = \sum_{j=0}^n \frac{(- n)_j (-x)_j}{(\alpha)_j} \frac{1}{j!} \Bigl( 1- \frac 1p\Bigr)^j = \frac{x^n}{(\alpha)_n} \Bigl(1-\frac 1p\Bigr)^n + O(x^{n-1})
\]	
They satisfy the orthogonality relation 
\begin{equation}\label{eq:ortho-meixner}
	(1-p)^{\alpha} \sum_{x=0}^\infty \frac{(\alpha)_x}{x!} p^x M_n(x;\alpha,p)M_m(x;\alpha,p) =  \Bigl(\frac{(\alpha)_n}{n!} p^n\Bigr)^{-1} \delta_{m,n}. 
\end{equation}
Consider the probability weights 
\[
	\rho_{p,\alpha}(x) = (1-p)^{-\alpha} \frac{p^x}{x!} (\alpha)_x \quad (x\in\N_0)
\] 
and the Hilbert space $\ell^2(\N_0,\rho_{p,\alpha})$. By the orthogonality relation~\eqref{eq:ortho-meixner}, the operator 
\[
	U:\ell^2(\N_0,w_\alpha)\to \ell^2(\N_0,\rho_{p,\alpha}),\quad f\mapsto Uf = \sum_{n=0}^\infty f(n) \frac{(\alpha)_n}{n!} \sqrt p^n M_n(\cdot;\alpha,p).
\]
is an isometry. Let $\mathcal M_n(x;\alpha,p)$ be the Meixner polynomials with leading coefficient $1$, then 
\[
	U e_n = \frac{1}{n!} \bigl(  \sqrt p - \sqrt p^{-1} \bigr)^{n} \mathcal M_n(\cdot;\alpha,p). 
\] 
Let us define $K^\# = U^{-1} k^\# U$, then we obtain operators with common domain the set of polynomials, that satisfy 
\begin{equation} \label{eq:kplusit-uni}
	K^+ \1 =0,\quad (K^+)^n \1 = (\sqrt p - \sqrt p^{-1})^n \mathcal M_n(\cdot;\alpha,p)
\end{equation} 
in analogy with $k^- e_0 =0$ and $(k^+)^n e_0 = n! e_n$. The following proposition identifies the operators $K^\#$. We define forward and backward difference operators 
\[
	\partial^+ f(x) = f(x+1) - f(x),\quad \partial^- f(x) = f(x-1) - f(x)
\] 
with $f(-1) =0$. The proposition is similar to Proposition~4.7 in Groenevelt \cite{Groenevelt2019}.

\begin{prop} \label{prop:rep-uni-meixner}
	For every polynomial $f:\N_0\to \C$, 
	\begin{align*}
		K^+ f(x) & = -  \frac{\sqrt p}{1-p}\Bigl( p(\alpha+x) f(x+1) + \frac1 p x f(x-1) -(\alpha + 2 x) f(x) \Bigr),\\
		K^- f(x) & = - \frac{\sqrt p}{1-p}\Bigl( (\alpha+x) \partial^+ f(x) + x\partial^-f(x)\Bigr),\\
		K^0 f(x) & = \frac\alpha 2 f(x) - \frac 1{1-p}\bigl( p(\alpha+x) \partial^+ f(x) + x \partial^- f(x)\bigr).
	\end{align*} 
\end{prop}

We remark that up to additive and multiplicative constants, the lowering and the neutral operators are the generators of linear birth-death chains on $\N_0$. As mentioned after Theorem~\ref{thm:rep-meixner}, those processes were already studied in connection with orthogonal polynomials by Karlin and McGregor\cite{karlin-mcgregor1958}; Miclo and Patie \cite{miclo-patie2019} provided intertwining relations between these linear birth-death chains and squared Bessel and Laguerre diffusions. Our representation clarifies the algebraic similarity of the processes: they admit the same formal expression with the $su(1,1)$ algebra. 

\begin{proof} [Proof of Proposition~\ref{prop:rep-uni-meixner}]
	The generating function of the Mexiner polynomials is \cite{HypergeometricOrthogonalPolynomials}
	\[
		G_t(x) = \sum_{n=0}^\infty \frac{(\alpha)_n}{n!}\, t^n M_n(x) = \Bigl(1- \frac t p \Bigr)^x (1-t)^{- \alpha - x}. 
	\] 
	We note 
	\begin{align*}
		\partial_+ G_t(x)  & = G_t(x+1) - G_t(x)  
				 = \Bigl( \frac{1- t/p }{1-t} - 1\Bigr) G_t(x) 
				 =  \frac{1- 1/ p}{1 - t} \, t G_t(x).
	\end{align*} 
	and 
	\[
		\partial_-  G_t(x) =  \Bigl( \frac{1- t }{1-t/p} - 1\Bigr) G_t(x)  =  \frac{1/p- 1}{1 - t/p} \, t G_t(x).
	\] 
	Therefore, in view of $(K^0 - \frac \alpha 2) M_n = n M_n$, we get 
	\begin{align*}
		\bigl( K^0 - \frac \alpha 2\bigr) G_t (x) & = \sum_{n=0}^\infty \frac{(\alpha)_n} {n!} t^n n M_n(x) = t \partial_t G_t(x) \\
		& = \Bigl( - \frac x  p\bigl( 1- \frac t p \bigr)^{-1} + (\alpha+ x) (1-t)^{-1} \Bigr) t G_t(x) \\
		& = - \frac{x}{1-p} \partial_- G_t(x) - \frac{p(\alpha+x)}{1-p}\partial_+ G_t(x). 
	\end{align*} 
	Turning to the lowering operators, we note first, $K^- M_0 =0$ and,  for $n\geq 1$, 
	\begin{align*}
		K^- \frac{(\alpha)_n}{n!} t^n M_n & = \bigl( \frac{t}{\sqrt p}\bigr)^n K^- U e_n = \bigl( \frac{t}{\sqrt p}\bigr)^n  U k^- e_n \\
		& = \bigl( \frac{t}{\sqrt p}\bigr)^n  (\alpha +n-1) U e_n \\
		& = \bigl( \frac{t}{\sqrt p}\bigr)^n (\alpha + n-1)  \frac{(\alpha)_{n-1}}{(n-1)!} {\sqrt p}^{n-1} M_{n-1} \\
		& = \frac{1}{\sqrt p} \frac{(\alpha)_{n}}{(n-1)!}\, t^n  M_{n-1}.
	\end{align*} 
	We change the summation index in $G_t$ from $n$ to $\ell = n-1$, note $(\alpha)_n = (\alpha)_{\ell+1} = (\alpha+\ell) (\alpha)_\ell$  and obtain
	\begin{align*}
		K^- G_t(x) & = \frac{1}{\sqrt p} \sum_{\ell=0}^\infty \frac{(\alpha)_{\ell}}{\ell!} (\alpha+ \ell) t^{\ell+1} M_{\ell}(x)  \\
		& = \frac{1}{\sqrt p} \bigl( \alpha t + t^2 \partial_t \bigr) G_t(x) \\
		& = \frac{1}{\sqrt p}\Bigl( \alpha t- \frac{x t^2}{p-t}+ \frac{(\alpha+x)t^2}{1-t} \Bigr) G_t(x)  \\
		& = \frac{1}{\sqrt p}\Bigl(- \frac{x}{1- t/p}+ \frac{\alpha+x}{1-t} \Bigr) tG_t(x)  \\
		& = - \frac{\sqrt p}{1-p} \bigl( x \partial_- G_t(x) + (\alpha + x) \partial_+ G_t(x)\bigr)
\end{align*} 
	The above relations hold true for all sufficiently small $t$, the expressions for $K^0$ and $K^-$ follow. 	The expression for the raising operator is proven by similar explicit computations that we omit. Alternatively, the reader may check the adjointness $\la f, K^+ g\ra = \la K^- f,g\ra$ on a dense subspace of polynomials in $\ell^2(\N_0,\rho_{p,\alpha})$. 
\end{proof}

\section{Proofs}\label{sec:proofs}
In this section we prove our main results, namely the three representations of the $su(1,1)$ current algebra, the Baker-Campbell-Hausdorff formula, and the explicit action of the unitaries $\exp(k^+(\xi) - k^-(\xi)) \exp( 2\mathrm i k^0(\theta))$ on exponential vectors. In addition, we prove an integration by parts formula for the Gamma random measure (Proposition~\ref{prop:gamma-ibp}) that is of interest in its own right.

\subsection{Representation in extended Fock space. Proof of Theorem~\ref{thm:fock-rep}} \label{sec:proof-extended-fock} 

We start with a lemma that helps prove the adjointness relations for the lowering and raising operators in extended Fock space. 

\begin{lemma} \label{lem:lambdan-seq}
	For all $n\in \N$ and non-negative or integrable $f:\mathbb X^{n+1}\to \C$, 
	\begin{equation}\label{eq:lambdan-seq}
		\int f \dd \lambda_{n+1} = \int \Bigl( \int f(x_1,\ldots, x_n,y) \alpha(\dd y) + \sum_{i=1}^n f(x_1,\ldots,x_n,x_i)\Bigr) \lambda_n(\dd \vect x).  
	\end{equation}
\end{lemma} 

\begin{proof}
	We have to evaluate $\sum_{\sigma \in \mathfrak S_{n+1}} \int f^{\sigma} \dd \alpha^{|\sigma|}$. The permutations with $\sigma(n+1)=n+1$ are in one-to-one correspondence with permutations in $\mathfrak S_n$, they contribute the first term on the right in~\eqref{eq:lambdan-seq}. Any permutation with $\sigma(n+1) = i \leq n$ admits a cycle decomposition with a single cycle containing $n+1$ and $i$ and remaining cycles containing neither $n+1$ nor $i$. Deleting $n+1$ from the cycle containing $n+1$ (e.g., $n+1\to i \to 3\to n+1$ becomes $i\to 3 \to i$), we obtain the cycle decomposition of a permutation $\tau \in \mathfrak S_n$ that has the same number of cycles. These permutations contribute the second term on the right in~\eqref{eq:lambdan-seq}. 
\end{proof} 

\begin{proof}[Proof of Lemma~\ref{lem:d-dense}]
	For $n\in \N$, let $V_n$ be the space of linear combinations of symmetrized tensor products of indicators of sets $B_j\in \mathcal X$ with $\alpha(B_j)<\infty$. It is enough to show that, for every $n\in \N$, the space $V_n$ is dense in $L^2_\mathrm s(\lambda_n)$. 

	Fix $n\in \N$. Cartesian products $B_1\times \cdots \times B_n$ of measurable sets with finite mass $\alpha(B_i)<\infty$ form a generating $\pi$-system of the product $\sigma$-algebra $\mathcal X^n$.  Therefore, by the functional monotone class theorem \cite[Theorem 2.12.9]{bogachev-vol1}, the linear combinations of the constant function and indicators of Cartesian products are dense, with respect to the supremum norm, in the space of bounded measurable functions from $\mathbb X^n$ to $\C$. It follows that the linear hull of $V_n$ and the constant function $1$ is dense, with respect to the supremum norm, in the space of bounded, measurable symmetric functions. 
	
	If $\alpha$ has finite total mass, then $\lambda_n(\mathbb X^n)<\infty$ as well and we have $\1\in V_n \subset L_\mathrm s^2(\lambda_n)$. Furthermore, uniform convergence implies $L^2(\lambda_n)$-convergence. Thus, every bounded function in $L_\mathrm s^2(\lambda_n)$ can be approximated by a function in $V_n$. As the bounded functions are dense in $L_\mathrm s^2(\lambda_n)$, we find that $V_n$ is dense.
	
	If $\alpha(\mathbb X) = \infty$, then the previous argument still applies to $L^2_\mathrm s(\lambda_n \1_{C^n})$, for every set $C\in \mathcal X$ with $\alpha(C)<\infty$ (hence $\lambda_n(C^n)<\infty$). As $\alpha$ is $\sigma$-finite, every function $f\in L_\mathrm s^2(\lambda_n)$ is the limit in $L^2$-norm of  a sequence $(f\1_{C_j^n})_{j\in \N}$ with $\alpha(C_j)<\infty$ and the proof of the  lemma is easily concluded. 
\end{proof} 

\begin{proof}[Proof of Theorem~\ref{thm:fock-rep}]
	By the expressions given above, the operators $k^\#(\varphi)$ map symmetrized tensor products to combinations of symmetrized tensor products, hence $\mathcal D$ is invariant. The common domain $\mathcal D$ is dense by Lemma~\ref{lem:d-dense}. Clearly $k^+(\varphi)$ and $k^-(\varphi)$ are linear in $\varphi$ while $k^-(\varphi)$ is antilinear in $\varphi$, moreover the vacuum is a unit vector annihilated by all lowering operators. The vacuum is cyclic because repeated application of raising operators gives symmetrized tensor products, whose linear hull is precisely the dense space $\mathcal D$. 
	
	For the commutation relations~\eqref{eq:commutation-relations-k}, we note
	\begin{align*}
		b(\varphi) a^\dagger (\psi) f_1\otimes_\mathrm s \cdots \otimes_\mathrm s f_n
		=  &2 \sum_{j=1}^n (\overline{\varphi}\psi f_j)\otimes_\mathrm s f_1\cdots \widehat f_j \cdots  \otimes_\mathrm s f_n \\
		& + 2 \sum_{1 \leq i < j \leq n} (\overline{\varphi}f_i f_j)\otimes_\mathrm s \psi\otimes_\mathrm s f_1 \cdots \widehat f_i\cdots \widehat f_j \cdots \otimes_\mathrm s f_n\\		
		 a^\dagger(\psi) b(\varphi) f_1\otimes_\mathrm s \cdots \otimes_\mathrm s f_n
		= &  2  \sum_{1\leq i < j \leq n} \psi \otimes_\mathrm s \bigl(\overline{\varphi} f_i f_j\bigr)\otimes_\mathrm s f_1\otimes_\mathrm s \cdots \widehat f_i \cdots \widehat f_j\cdots \otimes_\mathrm s f_n
	\end{align*} 
	hence 
	\begin{equation*}
		\bigl[ b(\psi), a^\dagger(\varphi)\bigr] f_1\otimes_\mathrm s \cdots \otimes_\mathrm s f_n
		 = 2 \widehat n(\overline{\psi} \varphi)\, f_1\otimes_\mathrm s \cdots \otimes_\mathrm s f_n 
	\end{equation*}
	and 
	\[
	   \bigl[ k^-(\psi), k^+(\varphi)\bigr] = \bigl[ a(\psi),a^\dagger (\psi)\bigr]+ \bigl[ b(\psi),a^\dagger (\psi)\bigr] = \la \psi,\varphi \ra + 2 \widehat n(\overline{\psi} \varphi) = 2 k^0(\overline{\psi} \varphi). 
	\]
	The other commutation relations, as well as Eq.~\eqref{eq:ccr}, are proven by similar explicit computations that we omit; see also {\'S}niady \cite[Theorem 4]{sniady2000}.
	
	The adjointness relation for the neutral operator is straightforward: $k^0(\varphi)$ is a multiplication operator that multiplies with $\sum_{i=1}^n \varphi(x_i) + \frac12 \int \varphi \dd \alpha$, its adjoint multiplies with the complex conjugate. For the raising and lowering operators, let $\varphi\in \mathcal C$ and $(F_n)$, $(G_n)$ in $\mathcal D$. We have 
	\begin{align*}
		& \frac{1}{(n+1)!}\int \overline{G_{n+1}}\,  (k^+(\varphi) F_n)\, \dd \lambda_{n+1}\\
		& \qquad =  \frac{1}{(n+1)!}\int \overline{G_{n+1}(x_1,\ldots,x_{n+1})} \Bigl( \sum_{i=1}^{n+1}\varphi(x_i) F_n(x_1,\ldots,\widehat x_i,\ldots,x_{n+1})\Bigr) \lambda_{n+1}(\dd \vect x) \\
		&\qquad = \frac{1}{n!} \int \overline{G_{n+1}(x_1,\ldots,x_{n+1})} \varphi(x_{n+1}) F_n(x_1,\ldots,x_n) \lambda_{n+1}(\dd \vect x).  
	\end{align*}  
	We apply Lemma~\ref{lem:lambdan-seq} to $f(x_1,\ldots,x_n,y) = \overline{G_{n+1}(x_1,\ldots,x_n,y)} \varphi(y) F_n(x_1,\ldots,x_n)$. The inner integral on the right of~\eqref{eq:lambdan-seq} becomes 
	\[
		\int \varphi(y) \overline{G_{n+1}(\vect x,y)} F_n(\vect x) \alpha(\dd y) 
		+ \sum_{i=1}^n \varphi(x_i) \overline{G_{n+1}(\vect x,x_i)} F_n(\vect x) 	
	\] 	
	in which we recognize $\overline{k^-(\overline{\varphi}) G_{n+1}} F_n$. Thus, we get
	\[
		\frac{1}{(n+1)!}\int \overline{G_{n+1}}\,  (k^+(\varphi) F_n)\, \dd \lambda_{n+1}
		= \frac{1}{n!}\int \overline{k^-(\overline{\varphi}) G_{n+1}} F_n \, \dd \lambda_n. 
	\] 
	Summation over $n\in \N_0$ yields the required adjointness relation. See also {\'S}niady \cite[Theorem 3]{sniady2000} or Lytvynov \cite[Section 2]{lytvynov2003JFA}.
\end{proof} 

\subsection{Baker-Campbell-Hausdorff formula. Proof of Theorem~\ref{thm:bch}} \label{sec:bch-proof} 

First we check that the product of exponentials $\exp(k^+(v))\exp(k^0(w))\exp(-k^-(v))$ in Theorem~\ref{thm:bch} is well defined on $\mathcal D$. 

\begin{lemma} \label{lem:serconv}
	Let $f \in \mathcal{D}$ and $v,w \in \mathcal C$. Then:
	\begin{enumerate}[(a)]
    	\item $k^-(v)^n f$ vanishes for all except finitely many $n$, hence $\exp( - k^-(v))f\in \mathcal D$.
       	\item $\exp(k^0(w))f$ is in $\mathcal{D}$. 
    	\item If $\sup |v|<1$, then  $\sum_{n=0}^\infty \frac{1}{n!} \bra{k^+(v)}^n f$ converges absolutely in norm. 
\end{enumerate}
\end{lemma}

The condition $\sup|v|<\infty$ is satisfied for $|v| = \tanh |\varphi|$ as $\varphi\in \mathcal C$ is bounded. 

\begin{proof}
It is enough to treat functions $f = k^+(\varphi_1) \cdots k^+(\varphi_m) \Psi$ with $\varphi_1, \ldots, \varphi_m \in \mathcal{C}$, $m \in \N_0$. Item (a) of the lemma follows from $\bra{k^-(\varphi)}^n f = 0$ for $n > m$. For item (b), notice that 
\begin{equation}
    \label{equation: exp k0 k+...k+}
   \exp\bigl(k^0(w)\bigr) k^+(\varphi_1) \cdots k^+(\varphi_m) \Psi  = \exp\Bigl(\frac{1}{2} \int w \dx \alpha\Bigr)  k^+(\varphi_1 e^{w}) \cdots k^+(\varphi_m e^{w}) \Psi \in \mathcal{D}. 
\end{equation}
For item (c), let $C > 0$, $0 < c < 1$ and $B \in \mathcal{X}$ be such that $\alpha(B) < \infty$ and $\abs{\varphi_j} \leq C \one_B$ for $j \in \set{1, \ldots, m}$ and $\abs{v} < c \one_B$. Then
\begin{multline*}
	\bigl( k^+(v)^n f\bigr)_\ell(x_1,\ldots, x_\ell) \\
	 = \delta_{n+m,\ell}\, \sum_{\sigma \in \mathfrak S_{n+m}} v(x_1)\cdots v(x_n) \varphi_{\sigma(n+1)}(x_{n+1})\cdots \varphi_{\sigma(n+m)}(x_{n+m}). 
\end{multline*}
Thus, 
\[
	\norm{k^+(v)^n f}^2 \leq \frac{1}{(n+m)!} \Bigl( (n+m)! c^n C^m \Bigr)^2 \lambda_{n+m}(B)
\] 
and 
\begin{equation*}
    \sum_{n = 0}^\infty \frac{1}{n!} \norm{(k^+(v))^n f}
    \leq  C^m \sum_{n = 0}^\infty \frac{c^n}{n!} \Bigl( (n+m)! (\alpha(B))_{(n+m)}\Bigr)^{1/2} <\infty. \qedhere
\end{equation*}
\end{proof}

For the proof of Theorem~\ref{thm:bch} we adapt Truax \cite{TR1985}, with extra care for unbounded operators. For $t \geq 0$, set 
\begin{align*}
    v_t := \one_{\{\varphi \neq 0\}} \frac{\varphi}{\abs{\varphi}} \tanh{t \abs{\varphi}}, \quad w_t := -2\log \cosh{t \abs{\varphi}}
\end{align*}
and define 
\begin{align*}
    S_t f :=  \exp\bigl( k^+(v_t)\bigr) \exp\bigl(k^0(w_t)\bigr) \exp\bigl( - k^-(v_t)\bigr) f, \qquad f \in \mathcal{D}.
\end{align*} 
The operator $S_t$ is well-defined on $\mathcal D$ by Lemma~\ref{lem:serconv}.
We show that $t\mapsto S_t f$ is norm-differentiable and solves the same differential equation as $\exp(t (k^+(\varphi) - k^-(\varphi)) f$. 

\begin{lemma}
	\label{lemma: product formula}
	The map $t \mapsto S_t f$ is norm-differentiable on $\R_+$, for every $f \in \mathcal{D}$, with derivative 
	\begin{multline}\label{eq: lemma 4}
  		\partial_t S_t f = \Big ( \e^{k^+(v_{t})}  k^+(\partial_t v_t) \e^{k^0(w_{t})} \e^{-k^-(v_{t})} 
		  + \e^{k^+(v_{t})} k^0(\partial_t w_t) \e^{k^0(w_{t})} \e^{-k^-(v_{t})} \\
		 - \e^{k^+(v_{t})} \e^{k^0(w_{t})} k^-(\partial_t v_t) \e^{-k^-(v_{t})} \Big ) f.
	\end{multline}
\end{lemma}

Note that $\partial_t v_t$, $\partial_t w_t \in \mathcal{C}$. The proof of the lemma is based on straightforward but somewhat tedious bounds and it is omitted.

The following commutation relations help us express the right side of \eqref{eq: lemma 4} with $k^+(\varphi)$ and $k^-(\varphi)$ only.
We remark that $k^+(\theta)$, $k^0(\theta)$ and $k^-(\theta)$ are closable because they have densely defined adjoints. We use the same letters for the closed operators. 

\begin{lemma} \label{lem:more-commutators}
	For all $v,w,\theta\in \mathcal C$ with $\sup|v|<\infty$, and all $f\in \mathcal D$, the following holds true:
	\begin{enumerate}[(a)]
		\item $	[k^0(\theta), \exp(k^+(v))] f= \exp(k^+(v)) k^+(\theta v)f$.
		\item $\exp(k^+(v))f$ is in the domain of the closure of $k^-(\theta)$ and 
		 \begin{equation*}
		 	\bigl[k^-(\theta), \e^{k^+(v)} \bigr]f = \bigl(- k^+(\overline{\theta}v^2 ) + 2 k^0(\overline{\theta}v )\bigr) \e^{k^+(v)}  f.
	 \end{equation*} 
		\item 	$\e^{k^0(w)} k^-(\theta) f = k^-(\theta\e^{-\overline{w}}) \e^{k^0(w)} f$.
	\end{enumerate} 
\end{lemma} 

\begin{proof}
	An induction over $n$, based on $[A, B^{n+1}] = [A, B] B^n + B [A, B^n]$ and the commutation relations for $k^0(\theta)$ and $k^+(v)$, yields 
	\[
		[k^0(\theta), k^+(v)^n] = n k^+(\theta v) k^+(v)^{n-1}
	\] 
	from which one deduces 
	\begin{equation} \label{eq:kzero-expkplus}
		[k^0(\theta), \exp(k^+(v))] = \exp(k^+(v)) k^+(\theta v). 
	\end{equation}
	Next we check by induction over $n$ that 
	\begin{equation} 
	\label{eq:kminus-kplus-ncommutator}
  	  [ k^-(\theta), k^+(v)^n] =  n(n - 1) k^+( \overline{\theta}v^2 ) k^+(v)^{n -2}  + 2n k^+(v)^{n-1} k^0(\overline{\theta} v). 
	\end{equation}
	For $n=1$ the previous relation reads $[k^-(\theta), k^+(v)] = 2 k^0(\overline{\theta} v)$, which holds true. For the induction step from $n$ to $n+1$, we write
	\begin{align*}
		[ k^-(\theta), k^+(v)^{n+1}] & = \bigl[k^-(\theta), k^+(v)\bigr] k^+(v)^n + k^+(v) \bigl[k^0(\theta), k^+(v)^n\bigr]\\
				& = 2 \Bigl[k^0(\overline{\theta} v), k^+(v)^n\Bigr] + 2 k^+(v)^{n} k_0(\overline{\theta} v) 
					+ k^+(v) \Bigl[k^0(\theta), k^+(v)^n\Bigr] \\
				& = 2n k^+(\overline{\theta} v^2) k^+(v)^{n-1} + 2 k^+(v)^{n} k_0(\overline{\theta} v) 
					+ k^+(v) \Bigl[k^0(\theta), k^+(v)^n\Bigr] \\
				& = n(n+1) k^+((\overline{\theta} v^2) k^+(v)^{n-1} +2(n+1) k^+(v)^{n} k^0(v \overline{\theta}).
	\end{align*}
	The inductive proof of~\eqref{eq:kminus-kplus-ncommutator} is complete. Now let $f\in \mathcal D$ and $N\in \N$. Eq.~\eqref{eq:kminus-kplus-ncommutator} yields 
	\begin{multline*}
		 k^-(\theta) \sum_{n=0}^N \frac{1}{n!} k^+(v)^n f \\
		 = \Bigl(\sum_{n=0}^N \frac{1}{n!} k^+(v)^n \Bigr) k^-(\theta) f + \Bigl(\sum_{n=0}^{N-2} \frac{1}{n!} k^+(v)^n \Bigr) k^+(\overline{\theta}v^2 ) f  
		 + 2 \Bigl(\sum_{n=0}^{N-1} \frac{1}{n!} k^+( v)^n \Bigr) k^0(\overline{\theta}v ) f.
	\end{multline*}
	The right side converges in norm as $N\to \infty$ by Lemma~\ref{lem:serconv}, hence so does the left side. It follows that $\exp(k^+(v)) f$ is in the domain of the closure of $k^-(\theta)$ and we obtain 
	 \begin{equation*}
	 	\bigl[k^-(\theta), \e^{k^+(v)} \bigr]f = \e^{k^+(v)} \bigl( k^+(\overline{\theta}v^2 ) + 2 k^0(\overline{\theta}v )\bigr) f.
	 \end{equation*} 
	In view of~\eqref{eq:kzero-expkplus}, we may also write
	 \begin{equation*}
	 	\bigl[k^-(\theta), \e^{k^+(v)} \bigr]f = \bigl(- k^+(\overline{\theta}v^2 ) + 2 k^0(\overline{\theta}v )\bigr) \e^{k^+(v)}  f.
	 \end{equation*} 
	For (c) it  is enough to consider functions $f = k^+(\varphi_1)\cdots k^+(\varphi_m) \Psi$ for $m\in \N$, $\varphi_1,\ldots, \varphi_m\in \mathcal C$. We write down the proof for $m=2$, the general case is similar. We have 
	\begin{align*}
		k^-(\theta) f & = k^-(\theta) k^+(\varphi_1) k^+(\varphi_2)  \Psi \\
		 & = \langle \theta, \varphi_1\rangle k^+(\varphi_2)\Psi + \langle \theta, \varphi_2\rangle k^+(\varphi_1)\Psi + 2 k^+(\varphi_1 \overline{\theta} \varphi_2) \Psi. 
	\end{align*}
	By~\eqref{equation: exp k0 k+...k+}, 
	\[
		\e^{k^0(w)} k^-(\theta) f = \e^{\int w\dd\alpha /2} \Bigl(\langle \theta, \varphi_1\rangle k^+(\varphi_2 \e^w)\Psi + \langle \theta, \varphi_2\rangle k^+(\varphi_1 \e^w)\Psi + 2 k^+(\varphi_1 \overline{\theta}  \varphi_2 \e^w) \Psi \Bigr). 
	\] 
	Similarly, 
	\begin{multline*}
		 k^-(\psi) \e^{k^0(w)} f \\
		  = \e^{\int w\dd\alpha /2} \Bigl(\langle \psi, \e^w \varphi_1\rangle k^+(\varphi_2 \e^w)\Psi + \langle \psi, \e^w \varphi_2\rangle k^+(\varphi_1 \e^w)\Psi + 2 k^+(\varphi_1\e^w  \overline{\psi} \varphi_2 \e^w) \Psi \Bigr).
	\end{multline*}
	The expressions are equal for $\psi = \exp(-\overline{w}) \theta$. 
\end{proof}

\begin{lemma}
	\label{lem:differential equation}
	Let $f \in \mathcal{D}$. Then, $S_t f$ is in the domain of the closure of $k^+(\varphi) -  k^-(\varphi)$, denoted by $k^+(\varphi) -  k^-(\varphi)$ as well. Moreover, 
	\begin{equation*}
    	\partial_t S_t f = \bra{k^+(\varphi) -  k^-(\varphi)} S_t f.
	\end{equation*}
\end{lemma}
	
\begin{proof} 
	In Lemma~\ref{lem:more-commutators}(b) we proved that $\exp(k^+(v)) f$ is in the domain of the closure of $k^-(\theta)$ by proving that $k^-(\theta) \sum_{n=0}^N n!^{-1} k^+(v)^n f$ converges in norm as $N\to \infty$. The same argument shows that $\exp(k^+(v)) f$ also belongs to the domain of the closure of $k^+(\varphi) - k^-(\varphi)$.
As $\exp( k^0(w_t))$ and $\exp(- k ^-(v_t))$ map $\mathcal D$ to $\mathcal D$ by Lemma~\ref{lem:serconv}, it follows that $S_t f$ is in th domain of the closure, for all $f\in \mathcal D$. 	
	By Lemma~\ref{lemma: product formula}, the derivative $\partial_t S_t f$ is a sum of three terms. The first is $k^+(\partial_t v_t) S_t f$. The second term is
	\[
		  \e^{k^+(v_{t})} k^0(\partial_t w_t) \e^{k^0(w_{t})} \e^{-k^-(v_{t})} f =\bigl(k^0(\partial_t w_t) - k^+(v_t \partial_t w_t) \bigr) S_t f,
	\] 
	where we have used Lemma~\ref{lem:more-commutators}(a).
	The third term is 
	\begin{align*}
		&- \e^{k^+(v_{t})} \e^{k^0(w_{t})} k^-(\partial_t v_t) \e^{-k^-(v_{t})}f\\
				& \qquad = - \e^{k^+(v_t)} k^-(\e^{-w_t} \partial_t v_t) \e^{k^0(w_{t})}\e^{-k^-(v_{t})}f \\
				& \qquad = \Bigl( - k^+(\e^{-w_t} \overline{\partial_t v_t} v_t^2) + 2 k^0(\e^{-w_t} \overline{\partial v_t} vt) - k^{-} (\e^{-w_t} \partial_t v_t)\Bigr) S_t f. 
	\end{align*} 	
	This time we have used Lemma~\ref{lem:more-commutators}(c) and then (b). Altogether, we get 
	\begin{multline*}
		\partial_t S_t f = \Bigl( k^+\bra{\partial_t v_t - (\partial_t w_t) v_t -v_t^2 \overline{(\partial_t v_t)}} \\
			 + k^0\bra{\partial_t w_t + 2v_t \overline{(\partial_t v_t)}} -k^-\bra{(\partial_t v_t) e^{-w_t}}\Bigr) S_t f. 
	\end{multline*}
	The proof is concluded with the differential equations 
\begin{align*}
    \partial_t v_t - (\partial_t w_t) v_t -v_t^2 \overline{(\partial_t v_t)} e^{-w_t} &= \varphi, \\
    \partial_t w_t + 2v_t \overline{(\partial_t v_t)} e^{-w_t} &= 0, \\
    (\partial_t v_t) e^{-w_t} &= \varphi.  \qedhere
\end{align*}
\end{proof}

Now we have all the ingredients to prove the Baker-Campbell-Hausdorff formula.

\begin{proof}[Proof of Theorem~\ref{thm:bch}]
	As noted after Eq.~\eqref{eq:udef}, the closure $A$ of $-\mathrm i (k^+(\varphi)- k^-(\varphi))$ is self-adjoint. By Stone's theorem, for every $f\in \mathcal D$, the function $t\mapsto \exp(\mathrm i t A) f$ is norm-differentiable with derivative 
	$\partial_t \exp(\mathrm i t A) f = \mathrm i A \exp(\mathrm i t A)f$. 
	Hence $g_t := \exp(t \bra{k^+(\varphi)-k^-(\varphi)}) f$ is norm-differentiable with derivative 
	\[
		\partial_t g_t = \bra{k^+(\varphi)-k^-(\varphi)} g_t. 
	\]
	The map $t\mapsto S_t f$ satisfies the same differential equation by Lemma~\ref{lem:differential equation}, moreover $S_0 f = g_0 = f$, therefore $g_t = S_t f$ for all $t\geq 0$. 
 In particular, $t = 1$ corresponds to the Baker-Campbell-Hausdorff formula. 
\end{proof}

\subsection{Gamma random measure. Proof of Theorem~\ref{thm:rep-laguerre}}\label{sec:proofs-gamma}

First we prove that the space of polynomials $\mathbb D$ is dense in $L^2(\mathbf M,\rho_{\alpha})$. We actually prove a slightly stronger statement:  the space of polynomials in the masses $\mu(B)$, with $\alpha(B)<\infty$, is dense.

\begin{proof} [Proof of Lemma~\ref{lem:polydense-gamma}] 
	Let $\mathscr U$ be the set of finite linear combinations of indicators $\1_{B}$, with non-negative coefficients, of sets with finite mass $\alpha(B)<\infty$. For $u\in \mathscr U$, let $\mathscr H_0$ be the set of exponential maps $F_u:\mathbf N\to \R_+$, $\mu\mapsto \exp( - \mu(u))$.  Further let $\mathscr H$ be the closure of $\mathscr H_0$, with respect to the supremum norm, of the linear combinations of exponentials. The set $\mathscr H_0$ is closed with respect to pointwise multiplication and the vector space $\mathscr H$ contains the constant function $1$  and is closed in $\mathscr L^\infty(\mathbf M)$. Therefore, by the functional monotone class theorem \cite[Theorem 2.12.9]{bogachev-vol1}, $\mathscr H$ contains all bounded functions that are measurable with respect to $\sigma(\mathscr H_0)$. In view of $\mu(B) = -  \log F_{\1_B}(\eta)$, the set $\mathscr H_0$ generates the full $\sigma$-algebra on $\mathbf M$. Thus, $\mathscr H$ contains all bounded measurable functions. 
	
	As a consequence, the linear combinations of exponentials $\mu\mapsto \exp( - \mu(u))$ are dense in $\mathscr L^\infty(\mathbf M)$ with respect to uniform convergence, hence a fortiori in $L^2(\mathbf M,\rho_{\alpha})$ with respect to  $L^2$-norm. To conclude, we note that the exponential in turn can be approximated in $L^2$-norm by polynomials: Simply consider a truncated exponential series and bound the remainder, using that there exists $\eps = \eps_u>0$ such that $	\int \exp(\eps \mu(u)) \rho_{\alpha}(\dd \mu) <\infty$. 

	We obtain that polynomials in the variables $\mu(u)$ are dense in $L^2(\mathbf M,\rho_\alpha)$. Every polynomial in $\mu(u)$, with $u$ a linear combination of indicators, is also a a polynomial in masses $\mu(B_i)$, $\alpha(B_i)<\infty$. The lemma and the stronger claim stated before this proof follow. 
\end{proof} 

Next we prove an identity for Gamma random measures that will help us prove the adjointness relations~\eqref{eq:adjoint-relations-k}. A straightforward integration by part shows that for every polynomial $f:\R\to \C$, given $a>0$, 
\[
	\int_0^\infty x f'(x) x^{a-1} \e^{-x} \dd x = \int_0^\infty (x - a) f(x) x^{a-1} \e^{-x} \dd x. 
\]
The generalization for Gamma random measures is as follows. 
We write $\E_{\rho_\alpha} [F(\mu)] = \int F(\mu) \rho_\alpha(\dd \mu)$. 

\begin{prop} \label{prop:gamma-ibp}
	Let $\xi$ be a Gamma random measure with law $\rho_\alpha$. 	For all $\varphi \in \mathcal C$ and all polynomials $F\in \mathbb D$, 
	\[
		\E_{\rho_\alpha} \Bigl[ \int \varphi(x) \frac{\delta F}{\delta \mu(x)}(\mu)\, \mu(\dd x)\Bigr] = \E_{\rho_\alpha} \Bigl[ \bigl(\mu(\varphi) - \alpha(\varphi)\bigr) F(\mu)\Bigr]. 
	\] 
\end{prop}

\begin{proof}
	We consider first exponential functions $F(\mu) = \exp(  \mu(f))$ and exploit that the Gamma random measure is compound Poisson. Let $\Pi = \sum_i \delta_{(X_i,Z_i)}$ be a Poisson point process on $\mathbb X\times \R_+$ with intensity measure $\alpha(\dd x) z^{-1}\exp( -z) \dd z$ and $\xi = \sum_i Z_i \delta_{X_i}$. By the Mecke equation for $\Pi$ \cite{LastPenroseLecturesOnThePoissonProcess}, 
	\[
		\E\Bigl[ \sum_i G( X_i,Z_i; \Pi) \Bigr] = \int_\mathbb X\Bigl\{ \int_0^\infty  \E\Bigl[ G(x,z;\Pi +\delta_{(x,z)})\Bigr] z^{-1}\e^{-z} \dd z\Bigr\} \alpha(\dd x)
	\]
	for all measurable $G: \mathbb X\times \R_+\times \mathbf N(\mathbb X\times \R_+)\to \R_+$. Let $\varphi$ and $f$ be non-negative test functions in $\mathcal C$. We assume that $||f||_\infty$ is sufficiently small so that $\E_{\rho_\alpha}[\exp( 2 \mu(f))] <\infty$. 	
	The Mecke equation applied to 
	\[
		G\Bigl(x,z;\sum_j \delta_{(x_j,z_j)}\Bigr) = z \varphi(x) f(x) \exp\Bigl(  \sum_j z_j f(x_j)\Bigr)
	\]
	yields
	\begin{equation}\label{eq:gamma-ibp1}
		\E\Bigl[ \xi(\varphi f) \e^{ \xi(f))}\Bigr] =		\int_\mathbb X\Bigl\{ \int_0^\infty  \E\Bigl[ z\varphi(x) f(x) \e^{ z f(x)} \e^{ \xi(f)}\Bigr] z^{-1}\e^{-z} \dd z\Bigr\} \alpha(\dd x).
	\end{equation}
	An integration by parts gives
	\[
		\int_0^\infty f(x) \e^{zf(x)}\e^{-z} \dd z = -1 + \int_0^\infty \e^{zf(x)} \e^{-z} \dd z. 
	\]
	We insert this identity on the right side in Eq.~\eqref{eq:gamma-ibp1} (after changing the order of integration). This gives rise to two terms. The first, coming from $-1$, stays as is. For the second term, we apply again the Mecke equation. At the end, we obtain 
	\[
		\E\Bigl[ \xi(\varphi f) \e^{ \xi(f))}\Bigr] = - \alpha(\varphi) \E\Bigl[ \e^{\xi(f)}\Bigr] + \E\Bigl[ \xi(\varphi) \e^{\xi(f)}\Bigr].
	\] 
	The same identity holds true for $f$ replaced by $sf$ with $s\in (-1,1)$. We expand in powers of $s$, equate coefficients on the left and right side, and obtain 
	\[
		\E\Bigl[ \xi(\varphi f) n \xi(f)^{n-1}\Bigr] = \E\Bigl[ \bigl( \xi(\varphi) - \alpha(\varphi)\bigr) \xi(f)^n\Bigr].
	\] 
	The compound Poisson process $\xi$ has law $\rho_\alpha$ and every polynomial is a linear combination of monomials $\mu \mapsto \mu(f)^n$. It follows that 
	\[
		\E_{\rho_\alpha} \Bigl[ \mu(\varphi f) P'\bigl(\mu(f)\bigr) \Bigr] = \E_{\rho_\alpha} \Bigl[ \bigl(\mu(\varphi) - \alpha(\varphi)\bigr) P\bigl( \mu(f)\bigr)\Bigr]
	\] 
	which is precisely the statement that we were looking for.
\end{proof} 

Let us show right away how Proposition~\ref{prop:gamma-ibp} implies the adjointness relations.

\begin{lemma} \label{lem:gamma-adjoints}
	The operators $K^\#(\varphi)$ defined in Section~\ref{sec:rep-gamma}  satisfy the adjointness relations~\eqref{eq:adjoint-relations-k}.
\end{lemma} 

\begin{proof} 
	Pick  $f,\varphi \in \mathcal C$ and $P$ a univariate polynomial. Set $F(\mu) = P(\mu(f))$ and $G(\mu) = P'(\mu(f))$. Notice that 
	\[
		 \frac{\delta^2 F}{\delta \mu(x)^2} (\mu) = f(x)^2  P''\bigl( \mu(x)\bigr)  =  f(x) \frac{\delta G}{\delta \mu(x)}(\mu)
	\] 
	Proposition~\ref{prop:gamma-ibp} applied to $\psi = \varphi f$ and $G$ thus yields 
	\begin{align}
		\E_{\rho_\alpha}\Bigl[ \int \varphi(x) \frac{\delta^2 F}{\delta \mu(x)^2} (\mu) \mu(\dd x) \Bigr] 
		& = \E_{\rho_\alpha}\Bigl[ \int \bigl( \mu(\varphi f) - \alpha(\varphi f) \bigr) P'\bigl( \mu(f)\bigr) \Bigr] \notag \\
		& = \E_{\rho_\alpha}\Bigl[ \int \varphi(x) \frac{\delta F}{\delta \mu(x)}(\mu) \, (\mu- \alpha)(\dd x) \Bigr], \label{eq:adjo}
	\end{align}
	The statement extends, by linearity, to all polynomials $F\in \mathbb D$, and it is equivalent to 
	\begin{equation}\label{eq:adjo1}
		\E_{\rho_\alpha}\Bigl[ k^-(\varphi) F (\mu)\Bigr] = \E_{\rho_\alpha}\Bigl[ \int \varphi(x) \frac{\delta F}{\delta \mu(x)} (\mu) \mu(\dd x) \Bigr] = \E_{\rho_\alpha}\Bigl[ \bigl( \mu(\varphi) - \alpha(\varphi)\bigr) F(\mu)\Bigr].
	\end{equation}
	(For the last equality we have used again Proposition~\ref{prop:gamma-ibp}.) 
	
	Now let $F$ and $G$ be two polynomials in $\mathbb D$ and $\varphi$ be a test function in $\mathcal C$. We apply Eq.~\eqref{eq:adjo1} to the function $F\overline G$, keep in mind the  product rule for the variational derivatives $\frac{\delta}{\delta \mu(x)}$, and get 
	\begin{multline} \label{eq:adjo2}
		\la G, k^-(\varphi) F\ra + \la k^-(\overline \varphi) G, F\ra + 2\E_{\rho_\alpha} \Bigl[ \int \varphi(x) \frac{\delta F}{\delta \mu(x)} \frac{\delta \overline G}{\delta \mu(x)} (\mu) \mu(\dd x)\Bigr] \\
		= \langle \bigl(\mu(\overline \varphi)- \alpha(\overline\varphi)\bigr) G, F\rangle. 
	\end{multline}
	On the other hand Proposition~\ref{prop:gamma-ibp} implies 
	\begin{multline*}
		\E_{\rho_\alpha}\Bigl[ \int \varphi(x) \overline{G(\mu)} \frac{\delta F}{\delta \mu(x)}(\mu) \mu(\dd x) \Bigr] 
		= - \E_{\rho_\alpha}\Bigl[ \int \varphi(x)  \frac{\delta \overline G}{\delta \mu(x)}(\mu) F(\mu) \mu(\dd x) \Bigr] \\
		+ \E_{\rho_\alpha}\Bigl[ \bigl( \mu(\varphi) - \alpha(\varphi)\bigr) \overline G F(\mu)\Bigr]
	\end{multline*}
	and then, by an argument similar to ~\eqref{eq:adjo}, 
	\begin{align*}
		&\E_{\rho_\alpha}\Bigl[ \int \varphi(x) \frac{\delta \overline{G}}{\delta \mu(x)} \frac{\delta F}{\delta \mu(x)}(\mu) \mu(\dd x) \Bigr] \\
		& \quad = - \E_{\rho_\alpha}\Bigl[ \int \varphi(x)  \frac{\delta^2 \overline G}{\delta \mu(x)}(\mu) F(\mu) \mu(\dd x) \Bigr] 
		+ \E_{\rho_\alpha}\Bigl[ \int \varphi(x) \frac{\delta \overline G}{\delta \mu(x)} (\mu) F(\mu) (\mu-\alpha)(\dd x) \Bigr]	 \\
		& \quad = - \la k^-(\overline \varphi) G, F\ra + \E_{\rho_\alpha}\Bigl[ \int \varphi(x) \frac{\delta \overline G}{\delta \mu(x)} (\mu) F(\mu) \mu(\dd x) \Bigr].
	\end{align*}
	We combine this equation with~\eqref{eq:adjo2} and get
	\begin{multline} 
		\la G, k^-(\varphi) F\ra  =  \la k^-(\overline \varphi) G, F\ra - 2\E_{\rho_\alpha} \Bigl[ \int \varphi(x) \frac{\delta \overline G}{\delta \mu(x)} (\mu) F(\mu) \mu(\dd x) ] \\
		+ \langle \bigl(\mu(\overline \varphi)- \alpha(\overline\varphi)\bigr) G, F\rangle. 
	\end{multline}
	The right side is precisely $\la K^+(\overline{\varphi}) F, G\ra$.  Thus we have proven the adjointness relation for the raising and lowering operators. The symmetry of the neutral operators is proven by similar computations that we omit. 	
\end{proof} 

Now we  have all we need to prove Theorem~\ref{thm:rep-laguerre}.

\begin{proof}[Proof of Theorem~\ref{thm:rep-laguerre}]
 (i) The set of polynomials  $\mathbb D$ is  dense in $L^2(\mathbf M,\rho_\alpha)$ by Lemma~\ref{lem:polydense-gamma} and one easily checks that $K^\pm(\varphi)$ and $K^0(\varphi)$ map polynomials to polynomials. 
 
 (ii) Cleary $K^-(\varphi)$ is antilinear in $\varphi$ and $K^0(\varphi)$ and $K^+(\varphi)$ are linear in $\varphi$. 
 
 (iii) It is enough to check the commutation relations on monomials $F(\mu) = \mu(f)^n$  (every other monomial is a linear combination of such simple powers, e.g., $\mu(f) \mu(g) = \frac1 4( \mu (f+g)^2 - \mu( f-g)^2)$).  Thus, let $F(\mu) = P(\mu(f))$ with $P(x) = x^n$. 
	
	We show first that the lowering operators commute. Explicit computations yield
 	\[
 		K^-(\varphi) F(\mu)  = \mu(\overline{\varphi} f^2) P''\bigl( \mu(f)\bigr) + \alpha(\overline{\varphi} f) P'\bigl( \mu(f)\bigr)
 	\]
  and 
  	 \begin{align*}
  		K^-(\theta) K^-(\varphi) F(\mu) = & \mu(\overline{\theta} f^2) \mu(\overline{\varphi} f^2)  P^{(4)}\bigl(\mu(f)\bigr) \\
  		&  +  \Bigl( 2\mu(\overline{\varphi}\overline{\theta} f^3) + \alpha(\overline \varphi f) \mu (\overline{\theta} f^2) +
  	\alpha(\overline \theta f) \mu (\overline{\varphi} f^2)	  \Bigr) P^{(3)}\bigl(\mu(f)\bigr) \\
  	&	+ \Bigl( \alpha(\overline{\theta} \overline{\varphi} f^2) + \alpha(\overline \theta f) \alpha(\overline{\varphi} f) \Bigr) P''\bigl(\mu(f)\bigr).
  \end{align*} 
  The right side is symmetric in $\theta$ and $\varphi$, hence $[K^-(\theta), K^-(\varphi)] F =0$. Turning to the commutation relation between the lowering and the neutral operators, we note that $K^0(\varphi)  = - K^-(\overline{\varphi}) + R(\varphi) + \frac12 \alpha(\varphi)$ with 
  \begin{equation*}
  	R(\varphi) F(\mu) = \int \varphi(x) \frac{\delta F}{\delta \mu(x)} \mu(\dd x) = \mu(\varphi f) P'\bigl( \mu(f)\bigr).
  \end{equation*}
  To lighten notation, we drop the variables $\mu(f)$ from the polynomial $P$ and its derivatives in the final expressions. 
  We compute 
  \begin{align*}
  	R(\varphi) K^-(\theta) F(\mu) & = \int \varphi(x) \frac{\delta}{\delta \mu(x)} \Bigl( \mu(\overline{\theta} f^2) P''\bigl(\mu(f)\bigr) + \alpha(\overline{\theta} f) P'\bigl(\mu(f)\bigr)  \Bigr) \mu(\dd x)\\
  	& = 
  	\mu(\varphi \overline \theta f^2) P'' + \mu(\overline\theta f^2) \mu(\varphi f) P^{(3)} + \alpha(\overline{\theta} f) \mu(\varphi f) P'' 
  \end{align*} 
  and 
  \begin{align*}
  	K^-(\theta) R(\varphi) F(\mu)  & = 
  		\int \overline{\theta(x)} \frac{\delta^2}{\delta \mu(x)^2} \mu(\varphi f) P'\bigl( \mu(f)\bigr) \mu(\dd x) \\
  	&\qquad \qquad 	+ \int \overline{\theta(x)} \frac{\delta}{\delta \mu(x)} \mu(\varphi f) P'\bigl( \mu(f)\bigr) \alpha(\dd x) \\
  	& = 
  	2 \mu (\overline \theta \varphi f^2) P'' + \mu(\varphi f) \mu(\overline{\theta} f^2) P^{(3)} \\ 
  	&\qquad \qquad + \alpha(\overline{\theta} \varphi f) P' + \alpha(\overline \theta f) \mu(\varphi f) P''.
  \end{align*}
  It follows that 
  \begin{align*}
  	[K^0(\varphi), K^-(\theta)] F(\mu) & = [R(\varphi), K^-(\theta)] F(\mu) \\
  	& = - \mu (\overline \theta \varphi f^2) P''- \alpha(\overline{\theta} \varphi) P' \\
  	& = - K^-(\theta  \overline \varphi) F(\mu).
  \end{align*} 
  The other commutation relations are proven by similar computations that we omit.
  
  (iv) The adjointness relations have been proven in Lemma~\ref{lem:gamma-adjoints}.  
  
  (v) The constant function $1$ has norm $1$ in $L^2(\mathbf M,\rho_\alpha)$ because $\rho_\alpha$ is a probability measure. It is annihilated by the lowering operators because the latter are differential operators.
  
  (vi) We show first that raising operators yield Laguerre polynomials. Let $\ell\in \N$ and $B_1,\ldots, B_\ell\in \mathcal X$ be disjoint sets with $\alpha(B_i)<\infty$. Further let $n_1,\ldots, n_\ell \in \N$. Remember the univariate operators $K^+$ from Proposition~\ref{prop:rep-uni-laguerre}. The operator used a scalar $\alpha>0$ as a parameter, but now $\alpha$ is a measure; we define $K_i^+$ as the univariate operator with scalar $\alpha(B_i)>0$. 
	Write $\mathbf 1$ for the constant function $\N_0\ni n\mapsto 1$. An induction over $n_1+\cdots + n_\ell$ yields 
	\[
		K^+(\1_{B_1})^{n_1}\cdots K^+(\1_{B_\ell})^{n_\ell} \1 (\eta) = \prod_{i=1}^{\ell} \bigl( (-K_i^+)^{n_i} \mathbf 1\bigr)\bigl( \mu(B_i)\bigr). 
	\] 
	The minus sign arises because we made slightly different choices when defining $K^\#$ and $K^\#(\varphi)$: In the univariate case, we prioritized the definition through unitary equivalence, for the current algebra we changed signs so that $K^+(\1_B)^n\1$ is a polynomial with positive leading coefficient. Next we note
	\[
		(-K_i^+)^{n_i}(\mu(B_i))^{n_i}
 = (-1)^{n_i} n_i! L_{n_i}^{(\alpha(B_i)-1)} \bigl( \mu(B_i) \bigr) = \mathscr L_{n_i}^{(\alpha(B_i)-1)}\bigl( \mu(B_i)\bigr).
	\]
	Eq.~\eqref{eq:kplusit-laguerre} in Theorem~\ref{thm:rep-laguerre} follows. In particular, the linear hull of $\1$ and the iterates $\prod_{i=1}^\ell K^+(\1_{B_i})^{n_i}\1$ contains all polynomials in occupation numbers $\mu(B)$, where $B$ runs through the sets of finite measure $\alpha(B)<\infty$. As noted at the beginning of this section, these polynomials are dense in $L^2(\mathbf M,\rho_\alpha)$, therefore the constant function $1$ is cyclic.
\end{proof}

\subsection{Pascal point process. Proof of Theorem~\ref{thm:rep-meixner}} \label{sec:proofs-pascal}
For the commutation relations for $K^\pm (\varphi)$, $K^+(\theta)$, it is convenient to introduce another family of operators $k^\#(\varphi)$, $\varphi \in \mathcal C$, defined on the space $\mathbb D$ of polynomials by
\begin{align*}
	k^-(\varphi) F(\eta) & =\frac{1}{\sqrt p} \int \overline{\varphi(x)}  F(\eta+\delta_x) (\alpha + \eta)(\dd x),\\
	k^+(\varphi) F(\eta) &= \sqrt p \int  \varphi(x) F(\eta-\delta_x) \eta(\dd x),\\
	k^0(\varphi) F(\eta) & = F(\eta) \int \varphi(x) \Bigl( \eta+ \frac12 \alpha \Bigr) (\dd x). 
\end{align*}
These operators are similar to their counterparts on the extended Fock space. 
Notice 
\begin{align*}
    K^-(\varphi) &= \frac{1}{c} \bra{ \sqrt p\, k^+(\overline{\varphi})  + \frac 1{\sqrt p}\, k^-(\varphi)- 2  k^0(\overline{\varphi})}, \\
     K^+(\varphi) &= \frac{1}{c} \bra{ \frac{1}{\sqrt p}\, k^+(\varphi)  + \sqrt p\, k^-(\overline{\varphi})- 2 k^0(\varphi)},\\
    K^0(\varphi) &=  \frac{1}{c} \bra{ - k^+(\varphi)  -k^-(\overline{\varphi}) + \bigl({\sqrt p}^{-1}+\sqrt p\bigr) k^0(\varphi)}.
\end{align*}

\begin{lemma} \label{lem:commrel-meixner}
	The operators $k^\#(\varphi)$, $\varphi\in \mathcal C$, map $\mathbb D$ to itself and on $\mathbb D$ they satisfy the commutation relations~\eqref{eq:commutation-relations-k}. 
\end{lemma} 

\begin{proof} 
	Consider a monomial $ F(\eta) = \eta(f_1)\cdots \eta(f_n)$, with $n\in \N$ and $f_1,\ldots,f_n\in \mathcal C$. Then 
	\begin{align*}
		k^+(\varphi) F(\eta)  & = \sqrt p \int \varphi(x) \prod_{i=1}^n \bigl( \eta(f_i) - f_i(x)\bigr) \eta(\dd x) \\
		& = \sqrt p \sum_{I\subset [n]} (-1)^{n-|I|} \Biggl( \prod_{i\in I} \eta(f_i)\Biggr) \eta\Bigl( \varphi\prod_{i\in [n]\setminus I} f_i \Bigr)
	\end{align*} 
	where $[n] = \{1,\ldots, n\}$. It follows that $k^+(\varphi) F$ is a polynomial. The constant function $\1$ is mapped to $\sqrt p \eta(\varphi)$, which is a polynomial as well. As every polynomial is a linear combination of $\1$ and monomials of degree $n\geq 1$, it follows that $k^+(\varphi) \mathbb D\subset \mathbb D$. The reasoning for $k^-(\varphi)$ and $k^0(\varphi)$ is similar and therefore omitted. 
	
	For the commutation relations, we note
\begin{align*}
	k^-(\varphi)k^+(\psi) F(\eta) & = \int\overline{\varphi(x)} \Bigl( \int \psi(y) F(\eta+\delta_x - \delta_y) (\eta+\delta_x)(\dd y) \Bigr)(\alpha+\eta)(\dd x), \\
	k^+(\psi)k^-(\varphi) F(\eta) & = \int \psi(y) \Bigl(\int \overline{\varphi(x)}f(\eta+\delta_x - \delta_y)(\alpha+\eta- \delta_y)(\dd x) \Bigr)  \eta(\dd y).
\end{align*} 
When we subtract the second line from the first, everything cancels out except the Dirac terms $\delta_x(\dd y)$ and $- \delta_y(\dd x)$ in the inner integrals. Hence
\begin{align*}
	[k^-(\varphi),k^+(\psi)] F(\eta) &= \Bigl( \int \overline{\varphi(x)} \psi(x) (\alpha+\eta)(\dd x) + \int \psi(y) \overline{\varphi(y)} \eta(\dd y) \Bigr) F(\eta) \\
	& = 2 k^0(\overline{\varphi}\psi) F(\eta).
\end{align*}
The other commutation relations are proven by similar explicit computations, we leave the details to the reader.
\end{proof} 

\begin{lemma} \label{lem:self-adj-meixner}
	The operators $k^\#(\varphi)$ satisfy the adjointness relations~\eqref{eq:adjoint-relations-k} in $L^2(\mathbf N,\rho_{p,\alpha})$. 
\end{lemma} 

\begin{proof} 
	The neutral operators $k^0(\varphi)$ is a multiplication operators which multiplies with $\eta(\varphi)+ \frac12 \alpha(\varphi)$. The adjoint multiplies with the complex conjugate, i.e., with $\eta(\overline{\varphi}) + \frac12 \alpha(\overline \varphi)$, hence $\langle F,k^0(\varphi) G\rangle = \langle k^0(\overline{\varphi}) F, G\rangle$. 
	
	For $k^\pm(\varphi)$, we exploit the relation 
	\begin{equation} \label{eq:papangelou}
		\iint F(x,\eta)  \eta(\dd x) \rho_{p,\alpha}(\dd \eta) = \iint F(x,\eta+\delta_x) p (\alpha+\eta)(\dd x) \rho_{p,\alpha}(\dd \eta). 
	\end{equation}
	It is valid for all measurable $F:\mathbf N\times \mathbb X\to \R_+$, see \cite[Proposition 3.1]{floreani-jansen-wagner2023algebraic}. Let $\varphi \in \mathcal C$ be non-negative and $f,g$ be non-negative polynomials, e.g., $f(\eta) = \eta(f_1)\cdots \eta(f_n)$ with $f_i:\mathbb X\to \R_+$ in $\mathcal C$, similarly for $g$. Eq.~\eqref{eq:papangelou} applied to $F(x,\eta) = \varphi(x) \overline{f(\eta+\delta_x)} g(\eta)$ yields $\la f, k^+(\varphi) g\rangle = \langle k^-(\varphi)f,g\ra$, compare \cite[Lemma 3.2]{floreani-jansen-wagner2023algebraic}. The relation is extended to all polynomials and all $\varphi \in \mathcal C$ by taking linear combinations. 
\end{proof} 

Now we are equipped for the proof of Theorem~\ref{thm:rep-meixner}. 

\begin{proof}[Proof of Theorem~\ref{thm:rep-meixner}]
	(i) The space $\mathbb D$ of polynomials is dense in $L^2(\mathbf N,\rho_{p,\alpha})$ by Lemma~\ref{lem:polydense-pascal} and the operators $K^\#(\varphi)$, as linear combinations of the operators $k^\#(\varphi)$, map $\mathbb D$ to itself by Lemma~\ref{lem:commrel-meixner}. 
	
	(ii) Clearly $K^+(\varphi)$ and $K^0(\varphi)$ are linear in $\varphi$ and $K^-(\varphi)$ is antilinear in $\varphi$. 
	
	(iii)  The commutation relations are deduced from Lemma~\ref{lem:commrel-meixner} as follows. The commutator $[K^-(\varphi), K^+(\theta)]$, times $c^2$, is a sum of three terms. First, 
\[
	\Bigl[ \sqrt p\, k^+(\overline{\varphi}) + \frac{1}{\sqrt p} k^-(\varphi),\frac{1}{\sqrt p}\, k^+(\theta)  + \sqrt p\, k^-(\overline{\theta}) \Bigr] = - 2 p k^0(\overline{\varphi}\theta) + \frac 2 p\, k^0(\overline{\varphi}\theta).
\]
Second, 
\[
	\Bigl[ \sqrt p\, k^+(\overline{\varphi}) + \frac{1}{\sqrt p} k^-(\varphi),- 2 k^0(\theta) \Bigr] 
	 = 2 \sqrt p\, k^+(\overline{\varphi}\theta) - \frac{2}{\sqrt p} k^-(\varphi\overline{\theta}).
\]
Third, 
\[
	\Bigl[ -2 k^0(\overline{\varphi}),\frac{1}{\sqrt p}\, k^+(\theta)  + \sqrt p\, k^-(\overline{\theta}) \Bigr] = - \frac{2}{\sqrt p}\, k^+(\overline{\varphi} \theta) + 2 \sqrt p\, k^-({\varphi} \overline{\theta}).
\]
The sum of these three expressions is 
\[
	- 2 c \bigl( k^+(\overline{\varphi} \theta) + k^-({\varphi}\overline{\theta}) \bigr) +2 c\bigl(\sqrt{p}^{-1} + \sqrt p) \bigr) k^0(\overline{\varphi} \theta\bigr) = 2 c^2 K^0(\overline{\varphi} \theta).
\] 
This completes the proof of $[K^-(\varphi), K^+(\theta)] = 2 K^0(\overline{\varphi}\theta)$. The proof of the other commutation relations is similar. 	

	(iv) The adjointness relations for $K^\#(\varphi)$ follow from the relations for $k^\#(\varphi)$, proven in Lemma~\ref{lem:self-adj-meixner}.
	
	(v) The constant function $\Psi = \1$ has norm $1$ in $L^2(\mathbf N,\rho_{p,\alpha})$ because $\rho_{p,\alpha}$ is a probability measure. The function $\Psi = \1$ is annihilated by all lowering operators $K^-(\varphi)$ because $\mathrm D^+_x \1(\eta) = \mathrm D_x^-\1(\eta) =0$ for all $x,\eta$. 
	
	(vi) We show first that raising operators yield Meixner polynomials. Let $\ell\in \N$ and $B_1,\ldots, B_\ell\in \mathcal X$ be disjoint sets with $\alpha(B_i)<\infty$. Further let $n_1,\ldots, n_\ell \in \N$. Remember the univariate operators $K^+$ from Proposition~\ref{prop:rep-uni-meixner}. The operator used a scalar $\alpha>0$ as a parameter, but now $\alpha$ is a measure; we define $K_i^+$ as the univariate operator with scalar $\alpha(B_i)>0$. 
	Write $\mathbf 1$ for the constant function $\N_0\ni n\mapsto 1$. An induction over $n_1+\cdots + n_\ell$ yields 
	\[
		K^+(\1_{B_1})^{n_1}\cdots K^+(\1_{B_\ell})^{n_\ell} \1 (\eta) = \prod_{i=1}^{\ell} \bigl( (-K_i^+)^{n_i} \mathbf 1\bigr)\bigl( \eta(B_i)\bigr). 
	\] 
	The minus sign arises because we made slightly different choices when defining $K^\#$ and $K^\#(\varphi)$: In the univariate case, we prioritized the definition through unitary equivalence, for the current algebra we changed signs so that $K^+(\1_B)^n\1$ is a polynomial with positive leading coefficient. By Eq.~\eqref{eq:kplusit-uni}, 
	\[
		(-K_i^+)^{n_i}(\eta(B_i))
 = (-1)^{n_i} \bigl( \sqrt p - {\sqrt p}^{-1}\bigr)^{n_i}\mathcal M_{n_i}\bigl( \eta(B_i);\alpha(B_i),p\bigr).
	\]
	The prefactor in front of the Meixner polynomial is precisely $c^{n_i}$ and Eq.~\eqref{eq:k+iterates} in Theorem~\ref{thm:rep-meixner} follows. In particular, the linear hull of $\1$ and the iterates $\prod_{i=1}^\ell K^+(\1_{B_i})^{n_i}\1$ contains all polynomials in occupation numbers $\eta(B)$, where $B$ runs through the sets of finite measure. By an argument completely analogous to the proof of Lemma~\ref{lem:polydense-gamma} given in Section~\ref{sec:proofs-gamma},  this space is dense in $L^2(\mathbf N,\rho_{p,\alpha})$. Thus $\Psi = \1$ is cyclic.
\end{proof} 

\appendix

\section{Representation in standard Fock space. Araki's scheme} \label{app:araki}

Araki's factorizable representations of a current group $G$ are of the form \cite{araki1969}
\[ 
	\mathcal U_g \mathscr E_f = \exp(\mathrm i c_g) \exp\Bigl( - \frac12 ||\varphi_g||^2 + \la \varphi_g, f\ra \Bigr) \mathscr E_{Q_g(f-\varphi_g)}.
\]
The unitaries $\mathcal U_g$ operate in a Fock space $\mathfrak H$  for which the single-particle Hilbert space has the structure of a fiber integral $\mathfrak h_1 = \int^\oplus \mathfrak h \dd \alpha$. The exponential states $\mathscr E_f$ are dense. The unitaries $Q_g$ form a representation of $G$ in the one-particle Hilbert space $\mathfrak h_1$, they are diagonal with respect to the fiber decomposition. The phase $c_g\in \R$ and the map $g\mapsto \varphi_g\in \mathfrak h_1$ have to satisfy some consistency conditions (cocycle conditions). The difficulty in Araki's approach is to find cocycles. 

Here we sketch how our representation in extended Fock space on $\mathbb X$ may be lifted to a representation in a Fock space on $\mathbb X\times \N$ and how it fits Araki's scheme.

\subsection{Isometric embedding into a standard Fock space. Chinese restaurant process}
We start with a known isomorphism of the extended Fock space with a subspace of a standard Fock space \cite{kondratiev-silva-streit-us1998, Lytvynov2003}.
Let $\mathfrak H$ be a standard Fock space for the one-particle Hilbert space $L^2(\mathbb X\times \N, \alpha \otimes \sum_{n=1}^\infty \frac 1 n \delta_n)$. We include a factorial in the norm, i.e., the norm of an element $\vect F = (F_m)_{m\in \N_0}$ of $\mathfrak H$ is 
\begin{multline*} 
	||\vect F||^2 = |F_0|^2 \\
	+ \sum_{m=1}^\infty \frac{1}{m!} \int \sum_{n_1,\ldots, n_m} \bigl|F_m\bigl( (x_1,n_1), (x_2,n_2), \ldots ( x_m,n_m) \bigr)\bigr|^2 
	\frac{1}{n_1}\cdots \frac{1}{n_m}\, \alpha^m(\dd \vect x).
\end{multline*}
In the picturesque language of the \emph{Chinese restaurant process} \cite[Chapter 3.1]{pitman2006combinatorial}, a configuration $(x_1,n_1),\ldots, (x_m,n_m)$ represents a seating plan in a restaurant with $m$ tables; table no.~$i$ is located at place $x_i$ and has $n_i$ customers. 
(More prosaically, think of $m$ piles with $n_i$ particles sitting atop each other on a location $x_i$.)
  The total number of customers (particles) is $n=n_1+\cdots + n_m$ and we may list the locations of the customers by repeating the entry $x_i$, for each $i=1,\ldots,m$,  $n_i$ times:
\[
	(x_1^{n_1} x_2^{n_2}\cdots x_m^{n_m}) = (x_1,\ldots x_1,\ldots, x_m,\ldots, x_m)\in \mathbb X^{n_1+\cdots + n_m}.
\] 
This induces a mapping from the extended Fock space $\mathfrak F$ to the Hilbert space $\mathfrak H$: Given an element $\vect f= (f_n)_{n\in \N_0}$  of $\mathfrak F$, we define a new sequence $\vect F = (F_m)_{m\in \N_0}$ 
by 
\[
	F_0 = f_0,\quad F_m\bigl( (x_1,n_1),\ldots, (x_m,n_m)\bigr) = f_{\sum n_i}( x_1^{n_1}\cdots x_m^{n_m}).
\]
The map $\mathfrak U: \mathfrak F\to \mathfrak H$, $\vect f\mapsto \vect F$ is norm-preserving. It maps exponential states to exponential states: The image of $\mathcal E_z \in \mathfrak F$ is given by $F_0 =0$ and
\begin{equation}\label{eq:expo-zf}
	F_m\bigl( (x_1,n_1),\ldots,(x_m,n_m)\bigr) = \prod_{i=1}^m z(x_i)^{n_i} = \prod_{i=1}^m f(x_i,n_i)
\end{equation} 
with $f(x,n) = z(x)^n$. That is,  $\mathfrak U \mathcal E_z = \mathscr E_f$ where $\mathscr E_f$ is an exponential state in $\mathfrak H$. 

\subsection{Lifting the representation}
 The operators $k^\#(\varphi)$ from Section~\ref{sec:res-fock} are lifted to operators $K^\#(\varphi)$ in $\mathfrak H$ with $\mathfrak U k^\#(\varphi) = K^\#(\varphi) \mathfrak U$, as follows. The neutral operator in $\mathfrak H$ is a multiplication operator
\begin{multline*}
	K^0(\varphi) F_m\bigl( (x_1,n_1),\ldots, (x_m,n_m)\bigr) \\ = \Bigl( \frac12 \alpha (\varphi) + \sum_{i=1}^m n_i \varphi(x_i)\Bigr) F_m\bigl( (x_1,n_1),\ldots, (x_m,n_m)\bigr).
\end{multline*}
The lowering operator is 
\begin{multline*}
	K^-(\varphi) F_m \bigl( (x_1,n_1),\ldots, (x_m,n_m)\bigr) \\
	  = \sum_{i=1}^m \overline{\varphi(x_i)} n_i F_m\bigl( (x_1,n_1)\ldots (x_i,n_i+1)\ldots (x_m,n_m)\bigr) \\
	  +  \int \overline{\varphi(x)} F_{m+1}\bigl( (x_1,n_1),\ldots, (x_m,n_m), (y,1)\bigr) \alpha(\dd y)
\end{multline*} 
(a new customer may join a table or open a new table). The raising operator is 
\begin{multline*}
	K^+(\varphi) F_m \bigl( (x_1,n_1),\ldots, (x_m,n_m)\bigr) \\
	  = \sum_{i=1}^m \varphi(x_i) n_i \Bigl( \1_{\{n_i\geq 2\}} F_m\bigl( (x_1,n_1)\ldots (x_i,n_i-1)\ldots (x_m,n_m)\bigr)\\
+ 	  \1_{\{n_i = 1\}} F_{m-1}\bigl( (x_1,n_1)\ldots \widehat{(x_i,1)} \ldots (x_m,n_m)\bigr)\Bigr)
\end{multline*} 
(if a solo diner leaves, the number $m$ of tables decreases). Exponentiating as in Theorem~\ref{thm:group}, we should obtain a family of unitary operators in $\mathfrak H$ and a representation of the current group for the universal cover of $SU(1,1)$. 
The action of 
\[
	\mathcal U(\xi,\theta) = \exp\bigl( K^+(\xi) - K^-(\xi)\bigr) \exp\bigl( 2 \mathrm i  K^0(\theta)\bigr)
\]
on exponential states $\mathscr E_f$ with $f(x,n) = z(x)^n$ (see~\eqref{eq:expo-zf}) is 
\begin{equation} \label{eq:uxitheta-lifted}
	\mathcal U(\xi,\theta) \mathscr E_f = \mathcal C_{\xi,\theta}(f) \mathscr E_{\widehat f},\quad \widehat f(x,n) = z_{\xi,\theta}(x)^n
\end{equation} 
with 
\begin{equation} \label{eq:cxitheta}
	\mathcal C_{\xi,\theta}(f) = \exp\Bigl( \mathrm i \int \theta \dd \alpha  - \int \log \Bigl( \cosh |\xi| + \e^{2 \mathrm i \theta } z \frac{\overline{\xi}}{|\xi|} \sinh|\xi| \Bigr) \dd \alpha \Bigr)
\end{equation}
and 
\begin{equation} \label{eq:zxitheta}
	z_{\xi,\theta}(x) = \frac{\e^{2 \mathrm i \theta(x)} z(x) + \frac{\xi(x)}{|\xi(x)|} \tanh|\xi(x)| }{1+  \e^{2\mathrm i \theta(x)} z(x) \frac{\overline{\xi(x)}}{|\xi(x)|} \tanh|\xi(x)|}.
\end{equation}

\subsection{Matching the lifted representation to Araki's scheme} We introduce a single-table, non-spatial Hilbert space $\mathfrak h = \ell^2(\N,\sum_{n=1}^\infty \frac 1n \delta_n)$, and operators
\[
	\mathrm k^+ f(n) = n \1_ {\{n\geq 2\}} f(n-1),\quad \mathrm k^- f(n) = n f(n+1),\quad \mathrm k^0 f(n) = n f(n). 
\] 
These operators satisfy the $su(1,1)$ commutation relations. The vacuum is $\psi_0(n) = \delta_{n,1}$, the Bargmann index is $m_0=2$ (because $\mathrm k^0 \psi_0 = \frac 2 2  \psi_0$). Set 
\[
	Q_{\xi,\theta} = \exp( \xi \mathrm k^+ - \overline \xi \mathrm k^-) \exp(2\mathrm i \theta \mathrm k^0).
\]
We use the same letters for the operators in $L^2(\mathbb X\times \N, \alpha\otimes \sum_{n=1}^\infty \frac{1}{n} \delta_n)$ that act on the $n$-variable only. Further set
\[
	\varphi_{\xi,\theta}(x,n) = \Bigl( - \e^{2 \mathrm i \theta} \frac{\xi}{|\xi|} \tanh|\xi| \Bigr)^n.
\]

\begin{prop} \label{prop:araki}
	Let $\xi,\theta \in \mathcal C$ be complex- and real-valued functions in $\mathcal C$, respectively, and $f\in L^2(\mathbb X\times \N, \alpha\otimes \sum_{n=1}^\infty \frac1n \delta_n)$. 	The unitary $\mathcal U(\xi,\theta)$ acts on the exponential states $\mathscr E_f \in \mathfrak H$ as 
	\begin{equation} \label{eq:lifted-u}
		     \mathcal U(\xi,\theta) \mathscr E_f = \exp \Bigl( \mathrm i \int \theta \dd \alpha - \frac 12 ||\varphi_{\xi,\theta}||^2 + \la \varphi_{\xi,\theta}, f\ra \Bigr) \mathscr E_{Q_{\xi,\theta} (f- \varphi_{\xi,\theta})}.
	\end{equation}
\end{prop} 

\begin{proof} [Proof sketch] 
	For $\xi \equiv 0$, the vector $\varphi_{\xi,\theta}$ vanishes and $\exp( 2\mathrm i \theta \mathrm k^0)$ is multiplication with $\exp(2 \mathrm i \theta(x))$. Thus, the right side of Eq.~\eqref{eq:lifted-u}  is $\exp(\mathrm i \alpha(\theta)) \mathscr E_{\exp( 2\mathrm i\theta) f}$, which is indeed the same as $\exp(2\mathrm i K^0(\theta))\mathscr E_f$. 
	For $\theta \equiv 0$, one may replace $\xi$ by $t\xi$ on both sides in~\eqref{eq:lifted-u},  differentiate with respect to $t$, and check that the differentiated equality holds true. The procedure is similar to the proof of Proposition 5.1 in our article \cite{floreani-jansen-wagner2023algebraic}, we omit the details. The full statement is deduced by concatenating $\mathcal U(\xi,\theta) = \mathcal U(\xi,0) \mathcal U(0,\theta)$. 
\end{proof} 

We complement the proof sketch by a consistency check with Eqs.~\eqref{eq:uxitheta-lifted}--\eqref{eq:zxitheta}. For $f(x,n ) = z(x)^n$ with $\sup|z|<1$, we compute 
\[
	\la \varphi_{\xi,\theta}, f\ra = \sum_{n=1}^\infty \frac1n \int \overline{\varphi_{\xi,\theta}(x,n)} f(x,n) \alpha(\dd x) 
	= - \int \log \Bigl( 1+ z \frac{\overline{\xi}}{|\xi|}\tanh|\xi|\Bigr) \dd\alpha. 
\] 
With the equality $1- (\tanh s)^2 = 1/ (\cosh s)^2$, we get 
\[
	- \frac 12 ||\varphi_{\xi,\theta}||^2 =  \frac 12 \int \log\bigl( 1- (\tanh|\xi|)^2\bigr) \dd \alpha =- \int \log (\cosh|\xi|) \dd \alpha.
\]
It follows that the exponential on the right side in~\eqref{eq:lifted-u} is equal to the factor $\mathcal C_{\xi,\theta}$ from~\eqref{eq:uxitheta-lifted} and~\eqref{eq:cxitheta}. It remains to check that $\mathscr E_{Q_{\xi,\theta}(f -\varphi_{\xi,\theta})}$ corresponds to $\prod_{j} z_{\xi,\theta}(x_j)^{n_j}$ with $z_{\xi,\theta}(x)$. 
For simplicity we only consider $\theta\equiv 0$ and real-valued $\xi$. Checking the desired equality boils down to a computation in $\ell^2(\N,\sum_{n=1}^\infty \frac1n \delta_n)$: for $t\in \R$, $z\in \C$ with $ |z|<1$, and 
\[
	z_t = \frac{z+ \tanh t}{1+ z\tanh t},\quad f_t(n) = z_t^n
\] 
one has 
\begin{align*}
	\frac{\dd}{\dd t} z_t^n = n z_t^{n-1} (1-z_t^2) = (\mathrm k^+ - \mathrm k^- ) f_t(n) + \delta_{n,1}.
\end{align*}
The special case $z = - \tanh s$ with fixed $s\in \R$ shows that $g_t(n) = (\tanh (t-s))^n$ satisfies the same differential equality, therefore $\partial_t (f_t - g_t) = (\mathrm k^+- \mathrm k^-) (f_t - g_t)$ and $f_t - g_t = \exp(t(\mathrm k^+- \mathrm k^-))(f_0-g_0)$. In the previous equality $g_0$ depends on $s$ and the equality holds true for every $s$, in particular, $s=-t$. For this choice $g_0(n) = (-\tanh t)^n$ and $g_t(n) =0$. Thus, we get
\[
	\e^{t   (\mathrm k^+ - \mathrm k^-)}(f_0 - g_0)  (n) = f_t(n) = z_t^n
\] 
which is essentially the desired equality.

\section{Generalized Lebesgue measure and \texorpdfstring{$SL(2,\R)$}{SL(2,R)}} \label{app:lebesgue}

Here we explain how our representations with the Gamma random measure relate to representations of the special linear group with the multiplicative measure. Let us first briefly describe the construction, following Tsilevich, Vershik and Yor \cite{tsilevich-vershik-yor2001}. We assume that the measure $\alpha$ on $\mathbb X$ is finite so that the Gamma random measure with law $\rho_\alpha$ has almost surely finite total mass, $\mu(\mathbb X)<\infty$. The measure $\mathscr L_\alpha$ is by definition the $\sigma$-finite measure on $\mathbf M$ that is absolutely continuous with respect to $\rho_\alpha$, with Radon-Nikodym derivative 
\[
	\frac{\dd \mathscr L_\alpha}{\dd \mathscr \rho_\alpha}(\mu) = \exp(\mu(\mathbb X)).
\]
The definition is motivated by a simple observation in the univariate case: The Lebesgue measure on $\R_+$ is absolutely continuous with respect to the Gamma distribution $\e^{-x}\dd x$ with parameter $1$, with Radon-Nikodym derivative $\exp(x)$.

The $SL(2,\R)$ current group over $\mathbb X$ consists of the bounded measurable maps taking values in $SL(2,\R)$, the real $2\times 2$ matrices with determinant $1$. The lower triangular subgroup consists of the elements  $g:\mathbb X\to SL(2,\R)$ of the form 
\[
	g(x) = \begin{pmatrix} 
				a(x)^{-1} & 0 \\
				b(x) & a(x) 
     \end{pmatrix}
\] 
with bounded measurable maps $a,b:\mathbb X\to \R$. The map $a$ is bounded away from zero. 
For $\mu \in \mathbf M$, the dilated measure $M_{a^2} \mu$ is given by $M_{a^2}\mu(B) = \int_B a^2(x) \mu(\dd x)$. For $F\in L^2(\mathbf M,\mathscr L_\alpha)$ and $g$ as above, define 
\[
	\mathcal U_g F (\mu) = \exp\Bigl( \int |\log a| \dd \alpha + \mathrm i \int a b \dd \mu \Bigr) F(M_{a^2} \mu).
\]
The map $g\mapsto \mathcal U_g$ is a unitary representation of the triangular subgroup of the $SL(2,\R)$ current group (\cite[Section 5]{tsilevich-vershik-yor2001} and references therein). The action of the full group is more complicated; in the univariate case it requires, in addition to the multiplication and dilation operators, an integral operator
\cite{tsilevich-vershik-yor2001}.

To establish the connection with our representation from Section~\ref{sec:rep-gamma}, we first notice that the spaces are unitarily equivalent. The map $F\mapsto \tilde F$ with $\tilde F(\mu) = \exp(\mu(\mathbb X)/2) F(\mu)$ is a unitary isomorphism from $L^2(\mathbf M,\rho_\alpha)$ onto $L^2(\mathbf M, \mathscr L_\alpha)$. In $L^2(\mathbf M,\rho_\alpha)$, the unitary $\mathcal U_g$ becomes
\[
	U_g F(\mu) = \exp\Bigl( \int |\log a| \dd \alpha + \mathrm i \int a b\, \dd \mu + \frac12 \int (1-a^2) \dd \mu \Bigr) F(M_{a^2} \mu)
\]
For $a(x)\equiv 1$, this is simply multiplication with $\exp( \mathrm i \mu(b))$ hence, in view of~\eqref{eq:gamma-fieldop}, this is the operator $\exp( \mathrm i (K^+(b)+ K^-(b) + 2K^0(b)))$. For $b(x) \equiv 0$ and $a(x) = \exp(t \varphi(x))$ with real-valued bounded $\varphi$, 
$U_g$ maps $F$ to 
\[
	F_t(\mu) = \exp\Bigl( t \int \varphi \dd \alpha  + \frac12 \int (1-\e^{2t\varphi}) \dd \mu \Bigr) F(M_{\exp(2t\varphi)} \mu).
\] 
An explicit computation for polynomial $F$ gives 
\begin{align*}
	\partial_t F_t(\mu) & = 2 \int \varphi(x)\frac{\delta F_t}{\delta \mu(x)} (\mu) \mu(\dd x) + \bigl( \alpha(\varphi) - \mu(\varphi) \bigr) F_t(\mu) \\
	& = \bigl( K^-(\varphi) - K^+(\varphi)\bigr) F_t (\mu)
\end{align*}
and hence, the unitary $U_g$ for $g$ the diagonal matrix with diagonal entries $\exp(-t \varphi(x))$, $\exp( t \varphi(x))$, corresponds to $\exp(t (K^-(\varphi)- K^+(\varphi))$. 

To conclude,  let $k^\pm,k^0$ be the $2\times 2$ matrices from Section~\ref{sec:22matrices}. The correspondence 
\begin{align*}
	\mathrm i (k^+ + k^- + 2  k^0) &\mapsto \begin{pmatrix}  0 & 0 \\ 1 & 0 \end{pmatrix}, && 	k^- - k^+ \mapsto \begin{pmatrix}  -1 & 0 \\ 0 & 1 \end{pmatrix},\\
	\mathrm i (k^+ + k^- - 2  k^0) &\mapsto \begin{pmatrix}  0 & 1 \\ 0 & 0 \end{pmatrix}
\end{align*}
extends to an isomorphism from the real Lie algebra $su(1,1)$ 
onto $sl(2,\R)$. 

Thus, the representations with the multiplicative measure $\mathscr L_\alpha$ agree with our representations with the Gamma random measure. We pass from one to the other by (1) mapping $L^2(\mathbf M,\mathscr L_\alpha)$ to $L^2(\mathbf M, \rho_\alpha)$, and (2) exploiting the isomorphism between the real Lie algebra $su(1,1)$ and $sl(2,\R)$.

\subsubsection*{Acknowledgements}
S.F.\ acknowledges financial support from the Deutsche For\-schungsgemeinschaft (DFG, German Research Foundation) under Germany’s Excellence Strategy– EXC-2047/1 – 390685813. S.J.\ and S.W.\ were supported under Germany's excellence strategy EXC-2111-390814868. S.F.\ and S.W.\ thank the  Hausdorff Institute for Mathematics (Bonn) for its hospitality during the Junior Trimester Program \textit{Stochastic modelling in life sciences} funded by the Deutsche Forschungsgemeinschaft (DFG, German Research Foundation) under Germany’s Excellence Strategy - EXC-2047/1 - 390685813. The authors thank F. Redig for his insights at the beginning of the project. S.W. thanks D. Span{\`o} for helpful discussions.



\begin{thebibliography}{10}

\bibitem{AccardiItoCalculus}
L.~Accardi and A.~Boukas, \emph{{It{\^o} Calculus and Quantum White Noise Calculus}}, Stochastic Analysis and Applications (Berlin, Heidelberg) (F.E. Benth, G.~Di~Nunno, T.~Lindstr{\o}m, B.~{\O}ksendal, and T.~Zhang, eds.), Springer Berlin Heidelberg, 2007, pp.~7--51.

\bibitem{accardi2009quantum}
\bysame, \emph{Quantum probability, renormalization and infinite-dimensional {$\ast$}-{L}ie algebras}, SIGMA Symmetry Integrability Geom. Methods Appl. \textbf{5} (2009), Paper 056, 31.

\bibitem{accardi2002renormalized}
L.~Accardi, U.~Franz, and M.~Skeide, \emph{Renormalized squares of white noise and other non-{G}aussian noises as {L}\'{e}vy processes on real {L}ie algebras}, Comm. Math. Phys. \textbf{228} (2002), no.~1, 123--150.

\bibitem{accardi-skeide2000}
L.~Accardi and M.~Skeide, \emph{On the relation of the square of white noise and the finite difference algebra}, Infin. Dimens. Anal. Quantum Probab. Relat. Top. \textbf{3} (2000), no.~1, 185--189.

\bibitem{araki1969}
H.~Araki, \emph{Factorizable representation of current algebra. {N}on commutative extension of the {L}\'{e}vy-{K}inchin formula and cohomology of a solvable group with values in a {H}ilbert space}, Publ. Res. Inst. Math. Sci. \textbf{5} (1969/70), 361--422.

\bibitem{BGL}
D.~Bakry, I.~Gentil, and M.~Ledoux, \emph{Analysis and geometry of {M}arkov diffusion operators}, Grundlehren der mathematischen Wissenschaften [Fundamental Principles of Mathematical Sciences], vol. 348, Springer, Cham, 2014.

\bibitem{bargmann1947}
V.~Bargmann, \emph{Irreducible unitary representations of the {L}orentz group}, Ann. of Math. (2) \textbf{48} (1947), 568--640.

\bibitem{barhoumi-ouerdiane-riahi2008}
A.~Barhoumi, H.~Ouerdiane, and A.~Riahi, \emph{Unitary representations of the {W}itt and {${\rm sl}(2,\mathbb{R})$}-algebras through renormalized powers of the quantum {P}ascal white noise}, Infin. Dimens. Anal. Quantum Probab. Relat. Top. \textbf{11} (2008), no.~3, 323--350.

\bibitem{berezansky1998jacobi-fields}
Y.M. Berezansky, \emph{On the theory of commutative {J}acobi fields}, Methods Funct. Anal. Topology \textbf{4} (1998), no.~1, 1--31.

\bibitem{berezansky-mierzejewski2000}
Y.M. Berezansky and D.A. Mierzejewski, \emph{The structure of the extended symmetric {F}ock space}, Methods Funct. Anal. Topology \textbf{6} (2000), no.~4, 1--13.

\bibitem{bogachev-vol1}
V.I. Bogachev, \emph{Measure theory}, vol.~1, Springer, Berlin, 2007.

\bibitem{boukas1988thesis}
A.~Boukas, \emph{Quantum stochastic analysis: {A} non-{B}rownian case}, ProQuest LLC, Ann Arbor, MI, 1988, Thesis (Ph.D.)--Southern Illinois University at Carbondale.

\bibitem{boukas1991}
\bysame, \emph{An example of a quantum exponential process}, Monatsh. Math. \textbf{112} (1991), no.~3, 209--215.

\bibitem{brif-vourdas-mann1996}
C.~Brif, A.~Vourdas, and A.~Mann, \emph{Analytic representations based on {${\rm SU}(1,1)$} coherent states and their applications}, J. Phys. A \textbf{29} (1996), no.~18, 5873--5885.

\bibitem{carinci-giardina-redig-book}
G.~Carinci, C.~Giardin{\`a}, and F.~Redig, \emph{Duality for {M}arkov processes: a {L}ie-algebraic approach}, Book, in preparation, 2024+.

\bibitem{chiribella-dariano-perinotti2006}
G.~Chiribella, G.M. D'Ariano, and P.~Perinotti, \emph{{Applications of the group $SU(1, 1)$ for quantum computation and tomography}}, Laser physics \textbf{16} (2006), no.~11, 1572--1581.

\bibitem{CIR}
J.C. Cox, J.E. Ingersoll, and S.A. Ross, \emph{A theory of the term structure of interest rates}, Econometrica \textbf{53} (1985), no.~2, 385--407.

\bibitem{transitionFunction}
S.N. Ethier and R.C. Griffiths, \emph{The transition function of a measure-valued branching diffusion with immigration}, Stochastic processes: A Festschrift in Honour of Gopinath Kallianpur, Springer, New York, 1993, pp.~71--79.

\bibitem{floreani-jansen-wagner2023algebraic}
S.~Floreani, S.~Jansen, and S.~Wagner, \emph{{Intertwinings for Continuum Particle Systems: An Algebraic Approach}}, Preprint. arXiv:2311.08763, 2023.

\bibitem{fuchs-book}
J.~Fuchs, \emph{Affine lie algebras and quantum groups: An introduction, with applications in conformal field theory}, Cambridge University Press, 1995.

\bibitem{gelfand-graev-vershik1971}
I.M. Gel'fand, M.I. Graev, and A.M. Vershik, \emph{Models of representations of current groups}, Representations of {Lie} groups and {Lie} algebras, {Proc}. {Summer} {Sch}., {Budapest} 1971, {Pt}. 2, 121-179 (1985)., 1985.

\bibitem{giardina-kurchan-redig2007}
C.~Giardin{\`a}, J.~Kurchan, and F.~Redig, \emph{{Duality and exact correlations for a model of heat conduction}}, J. Math. Phys. \textbf{48} (2007), no.~3, 033301.

\bibitem{giardina-kurchan-redig-vafayi2009}
C.~Giardin{\`a}, J.~Kurchan, F.~Redig, and K.~Vafayi, \emph{{Duality and hidden symmetries in interacting particle systems}}, J. Stat. Phys. \textbf{135} (2009), no.~1, 25--55.

\bibitem{glimm-jaffe}
J.~Glimm and A.~Jaffe, \emph{Quantum physics. a functional integral point of view}, second ed., Springer-Verlag, New York, 1987.

\bibitem{Groenevelt2019}
W.~Groenevelt, \emph{{Orthogonal Stochastic Duality Functions from Lie Algebra Representations}}, J. Stat. Phys. \textbf{174} (2019), no.~1, 97--119.

\bibitem{hudson-parthasarathy1984}
R.L. Hudson and K.R. Parthasarathy, \emph{Quantum {I}to's formula and stochastic evolutions}, Comm. Math. Phys. \textbf{93} (1984), no.~3, 301--323.

\bibitem{ito-kubo1988poisson-gauss}
Y.~Ito and I.~Kubo, \emph{Calculus on {G}aussian and {P}oisson white noises}, Nagoya Math. J. \textbf{111} (1988), 41--84.

\bibitem{jansen-kurt2014}
S.~Jansen and N.~Kurt, \emph{{On the notion(s) of duality for Markov processes}}, Probab. Surv. \textbf{11} (2014), 59--120.

\bibitem{kac-book}
V.G. Kac, \emph{Infinite-dimensional {L}ie algebras}, Cambridge University Press, 1990.

\bibitem{karlin-mcgregor1958}
S.~Karlin and J.~McGregor, \emph{Linear growth birth and death processes}, J. Math. Mech. \textbf{7} (1958), 643--662.

\bibitem{kipnis-marchioro-presutti1982}
C.~Kipnis, C.~Marchioro, and E.~Presutti, \emph{{Heat flow in an exactly solvable model}}, J. Statist. Phys. \textbf{27} (1982), no.~1, 65--74.

\bibitem{HypergeometricOrthogonalPolynomials}
R.~Koekoek, P.A. Lesky, and R.F. Swarttouw, \emph{{Hypergeometric Orthogonal Polynomials and their q-Analogues}}, Springer Monographs in Mathematics, Springer, Berlin, 2010.

\bibitem{kondratiev-silva-streit-us1998}
Y.G. Kondratiev, J.L. {da Silva}, L~Streit, and G.F. Us, \emph{Analysis on {P}oisson and gamma spaces}, Infin. Dimens. Anal. Quantum Probab. Relat. Top. \textbf{1} (1998), no.~1, 91--117.

\bibitem{kondratiev-lytvynov2000gamma}
Y.G. Kondratiev and E.~Lytvynov, \emph{Operators of gamma white noise calculus}, Infin. Dimens. Anal. Quantum Probab. Relat. Top. \textbf{3} (2000), no.~3, 303--335.

\bibitem{last2016}
G.~Last, \emph{Stochastic analysis for {P}oisson processes}, Stochastic analysis for {P}oisson point processes. Malliavin calculus, Wiener-It{\^o} chaos expansions and stochastic geometry (G.~Peccati and M.~Reitzner, eds.), Bocconi Springer Ser., vol.~7, Bocconi Univ. Press, 2016, pp.~1--36.

\bibitem{LastPenroseLecturesOnThePoissonProcess}
G.~Last and M.~Penrose, \emph{{Lectures on the Poisson Process}}, Institute of Mathematical Statistics Textbooks, Cambridge {U}niversity {P}ress, 2017.

\bibitem{lindblad-nagel1970}
G.~Lindblad and B.~Nagel, \emph{Continuous bases for unitary irreducible representations of {${\rm SU}(1,\,1)$}}, Ann. Inst. H. Poincar\'{e} Sect. A (N.S.) \textbf{13} (1970), 27--56.

\bibitem{Lytvynov2003}
E.~Lytvynov, \emph{{Orthogonal decompositions for {L}\'{e}vy processes with an application to the gamma, Pascal, and Meixner processes}}, Infin. Dimens. Anal. Quantum Probab. Relat. Top. \textbf{6} (2003), no.~1, 73--102.

\bibitem{lytvynov2003JFA}
\bysame, \emph{{Polynomials of Meixner's type in infinite dimensions—Jacobi fields and orthogonality measures}}, J. Funct. Anal. \textbf{200} (2003), no.~1, 118--149.

\bibitem{lytvynov2004square-of-white-noise}
\bysame, \emph{The square of white noise as a {J}acobi field}, Infin. Dimens. Anal. Quantum Probab. Relat. Top. \textbf{07} (2004), 619--629.

\bibitem{meyer2006quantum}
P.~Meyer, \emph{{Quantum Probability for Probabilists}}, 2 ed., Lecture Notes in Mathematics, vol. 1538, Springer, Berlin, 1995.

\bibitem{miclo-patie2019}
L.~Miclo and P.~Patie, \emph{On a gateway between continuous and discrete {B}essel and {L}aguerre processes}, Ann. H. Lebesgue \textbf{2} (2019), 59--98.

\bibitem{MalliavinCalculus}
D.~Nualart, \emph{The {M}alliavin calculus and related topics}, second ed., Probability and its Applications (New York), Springer-Verlag, Berlin, 2006.

\bibitem{perelomov1972}
A.M. Perelomov, \emph{Coherent states for arbitrary {L}ie group}, Comm. Math. Phys. \textbf{26} (1972), 222--236.

\bibitem{pitman2006combinatorial}
J.~Pitman, \emph{{Combinatorial stochastic processes}}, Lecture Notes in Mathematics, vol. 1875, Springer, Berlin, 2006, Lectures from the 32nd Summer School on Probability Theory held in Saint-Flour, July 7--24, 2002, With a foreword by Jean Picard.

\bibitem{reedSimonII}
M.~Reed and B.~Simon, \emph{Methods of modern mathematical physics. {II}. {F}ourier analysis, self-adjointness}, Academic Press [Harcourt Brace Jovanovich, Publishers], New York-London, 1975.

\bibitem{riesz1933}
M.~Riesz, \emph{Sur le probl{\`e}me des moments et le th{\'e}or{\`e}me de {Parseval} correspondant.}, Acta Litt. Sci. Szeged \textbf{1} (1923), 209--225.

\bibitem{ruehl-yunn1976}
W.~R\"{u}hl and B.C. Yunn, \emph{Representations of the universal covering group of {${\rm SU}(1,1)$} and their bilinear and trilinear invariant forms}, J. Mathematical Phys. \textbf{17} (1976), no.~8, 1521--1530.

\bibitem{schoutens}
W.~Schoutens, \emph{{Stochastic processes and orthogonal polynomials}}, Lecture Notes in Statistics, vol. 146, Springer, New York, 2000.

\bibitem{schutz1994non}
G.~Sch{\"u}tz and S.~Sandow, \emph{{Non-Abelian symmetries of stochastic processes: Derivation of correlation functions for random-vertex models and disordered-interacting-particle systems}}, Phys. Rev. E \textbf{49} (1994), no.~4, 2726.

\bibitem{scully-zubairy-book}
M.O. Scully and M.S. Zubairy, \emph{Quantum optics}, Cambridge {U}niversity {P}ress, 1997.

\bibitem{sniady2000}
P.~\'{S}niady, \emph{Quadratic bosonic and free white noises}, Comm. Math. Phys. \textbf{211} (2000), no.~3, 615--628.

\bibitem{CanonicalCorrelations}
D.~Span\`o and A.~Lijoi, \emph{{Canonical correlations for dependent gamma processes}}, Preprint. arXiv:1601.06079, 2016.

\bibitem{stannat2003}
W.~Stannat, \emph{Spectral properties for a class of continuous state branching processes with immigration}, J. Funct. Anal. \textbf{201} (2003), no.~1, 185--227.

\bibitem{sturm2020algebraic}
A.~Sturm, J.M. Swart, and F.~V\"{o}llering, \emph{The algebraic approach to duality: an introduction}, {Genealogies of Interacting Particle Systems}, Lect. Notes Ser. Inst. Math. Sci. Natl. Univ. Singap., vol.~38, World Sci. Publ., Hackensack, NJ, 2020, pp.~81--150.

\bibitem{TR1985}
D.R. Truax, \emph{Baker-{C}ampbell-{H}ausdorff relations and unitarity of {$\mathrm{SU}(2)$} and {$\mathrm{SU}(1,1)$} squeeze operators}, Physical Review D (3) \textbf{31} (1985), no.~8, 1988--1991.

\bibitem{tsilevich-vershik-yor2001}
N.~Tsilevich, A.~Vershik, and M.~Yor, \emph{An infinite-dimensional analogue of the {L}ebesgue measure and distinguished properties of the gamma process}, J. Funct. Anal. \textbf{185} (2001), no.~1, 274--296.

\bibitem{vershik-graev2009}
A.M. Vershik and M.I. Graev, \emph{Integral models of representations of current algebras of simple {L}ie groups}, Uspekhi Mat. Nauk \textbf{64} (2009), no.~2(386), 5--72, Translation in Russian Math. Surveys {\bf 64} (2009), no.~2, 205–271.

\bibitem{vershik-graev2011}
\bysame, \emph{A {P}oisson model of the {F}ock space and representations of current groups}, Algebra i Analiz \textbf{23} (2011), no.~3, 63--136, Translation in St. Petersburg Math. J. {\bf 23} (2012), no.~3, 459--510.

\bibitem{vershik-gelfand-graev1973}
A.M. Ver\v{s}ik, I.M. Gel'fand, and M.I. Graev, \emph{Representations of the group {${\rm SL}(2,\,{\bf R})$}, where {${\bf R}$} is a ring of functions}, Uspehi Mat. Nauk \textbf{28} (1973), no.~5(173), 83--128, Translation in Russian Math. Surveys {\bf 28} (1973), no. 5, 87–132.

\end{thebibliography}

\providecommand{\bysame}{\leavevmode\hbox to3em{\hrulefill}\thinspace}
\providecommand{\MR}{\relax\ifhmode\unskip\space\fi MR }
\providecommand{\MRhref}[2]{%
  \href{http://www.ams.org/mathscinet-getitem?mr=#1}{#2}
}
\providecommand{\href}[2]{#2}

\end{document}